\input ./style/arxiv-general.cfg
\documentclass[aop,MSNbibl,seceqn,dvips]{arximspdf}
\makeatletter
   \@ifpackageloaded{graphicx}{}{\usepackage{graphicx}}
\makeatother

\doi{10.1214/14-AOP918} 
\volume{43}
\issue{4}
\pubyear{2015}
\firstpage{1731}
\lastpage{1776}

\makeatletter
\newcommand{\eqref}[1]{(\ref{#1})}
\newcommand{\bZ}{\mathbb{Z}}
\newcommand{\cN}{\mathcal{N}}
\newcommand{\cE}{\mathcal{E}}
\renewcommand{\P}{\mathbf{P}}
\newcommand{\mB}{\bar{B}}
\newcommand{\lsquare}{\mbox{\scalebox{2}[1]{$\square$}}}
\newcommand{\pred}{\mathtt{Pred}}
\newcommand{\burnin}{\mathtt{burnin}}
\newcommand{\cross}{\mathtt{cross}}
\newcommand{\ind}{\mathbf{1}}
\newcommand{\mmod}{\operatorname{mod}}
\newcommand{\seeds}{\mathtt{Seeds}}
\newcommand{\add}{\textit{1~Or~3}}
\newcommand{\webxor}{\textit{Web-Xor}}
\newcommand{\pig}{\textit{Piggyback}}
\newcommand{\pbox}{\textit{Perturbed Box 13}}
\newcommand{\eqd}{\stackrel d=}
\newcommand{\df}{\emph}
\newtheorem{theorem}{Theorem}[section]
\newtheorem{lemma}[theorem]{Lemma}
\newtheorem{prop}[theorem]{Proposition}
\newtheorem{cor}[theorem]{Corollary}
\newtheorem{conjecture}[theorem]{Conjecture}
\makeatother

\begin{document}

\begin{frontmatter}

\title{Percolation and disorder-resistance in cellular~automata}
\runtitle{Cellular automata}

\begin{aug}
\author[A]{\fnms{Janko} \snm{Gravner}\corref{}\thanksref{T1}\ead[label=e1]{gravner@math.ucdavis.edu}}
\and
\author[B]{\fnms{Alexander E.} \snm{Holroyd}\ead[label=e2]{holroyd@microsoft.com}}
\runauthor{J. Gravner and A.~E. Holroyd}
\affiliation{University of California and Microsoft Research}
\address[A]{Department of Mathematics\\
University of California\\
Davis, California 95616\\
USA\\
\printead{e1}} 
\address[B]{Microsoft Research\\
1 Microsoft Way\\
Redmond, Washington 98052\\
USA\\
\printead{e2}}
\end{aug}
\thankstext{T1}{Supported in part by NSF Grant DMS-02-04376 and the Republic of Slovenia's Ministry of Science program P1-285.}

\received{\smonth{5} \syear{2013}}
\revised{\smonth{1} \syear{2014}}

%
\begin{abstract}
We rigorously prove a form of disorder-resistance for a class of
one-dimensional cellular automaton rules, including some that arise as
boundary dynamics of two-dimensional solidification rules. Specifically,
when started from a random initial seed on an interval of length $L$, with
probability tending to one as $L\to\infty$, the evolution is a \textit{replicator}.
That is, a region of space--time of density one is filled
with a
spatially and temporally periodic pattern, punctuated by a finite set of
other finite patterns repeated at a fractal set of locations. On the other
hand, the same rules exhibit provably more complex evolution from some seeds,
while from other seeds their behavior is apparently chaotic. A principal
tool is a new variant of percolation theory, in the context of
additive cellular automata from random initial states.
\end{abstract}

%
\begin{keyword}[class=AMS]
\kwd{60K35}
\kwd{37B15}
\end{keyword}
\begin{keyword}
\kwd{Additivity}
\kwd{cellular automaton}
\kwd{replicator}
\kwd{quasireplicator}
\kwd{ether}
\kwd{percolation}
\end{keyword}

\end{frontmatter}

\section{Introduction}\label{sec1}

Cellular automata (CA) started from \textit{seeds}, that is, finite
perturbations
of a quiescent state, have been the subject of much empirical analysis,
starting with \cite{Wol1}. The observed behavior falls roughly into four
categories: (a)~the perturbation remains \textit{localized} in the
sense that
it never affects sites outside a bounded interval; (b) a \textit{periodic}
structure develops and spreads; (c) a \textit{replicating} (also called
\textit{nested} or \textit{fractal}) structure develops, with a
recursive (but
sometimes complicated) description; (d) unpredictable \textit{chaotic} (or
\textit{complex}) growth generates a space--time configuration with apparent
characteristics of random fields. Many CA are capable of behavior in
multiple categories depending on the choice of seed, and this is true even
for some of the very simplest one-dimensional CA. An example is the
\textit{Exactly 1} rule, in which a cell is alive whenever exactly
one of
itself and its two neighbors were alive at the previous generation.
\textit{Exactly 1} is capable of periodic, replicating, and chaotic behavior
for different seeds; see \cite{GG3}.

If a particular CA is capable of chaotic behavior from some initial
seed, it
appears natural to conclude, by analogy with the second law of
thermodynamics, that such behavior should be generic for that CA, in the
sense that almost all sufficiently long seeds yield chaotic evolution.
Shadowing results from dynamical systems \cite{Pil}, with their general
message of stability of chaotic trajectories, would also tend to support
such a conclusion. Indeed, strong empirical evidence confirms that chaotic
behavior is prevalent for many CA including \textit{Exactly 1}; see
\cite{GG3}.

In this article, we exhibit a class of one-dimensional CA rules for
which we
rigorously prove that the \emph{opposite} conclusion holds. Typical (random)
long seeds self-organize into replicating structures, while exceptional seeds
yield more complex behavior, including apparently chaotic evolution.

We focus on one-dimensional range-$2$ CA rules with $3$ states
(although our
techniques in principle apply to more general one-dimensional rules).
Thus, the
configuration of the CA at time $t\in\{0,1,2,\ldots\}$ is an element
$\xi_t=(\xi_t(x))_{x\in\bZ}$ of $\{0,1,2\}^\bZ$, and for a given initial
configuration $\xi_0$, the evolution is given by
\[
\xi_{t+1}(x)=f \bigl(\xi_t(x-2),\xi_t(x-1),
\xi_t(x),\xi_t(x+1),\xi _t(x+2) \bigr)
\]
for all $x,t$ and a fixed
function $f$ (the CA rule). (In many cases, the dependence on $\xi_t$ will
actually be restricted to the range-$1$ neighborhood $x-1,x,x+1$.) We
sometimes write $\xi(x,t)=\xi_t(x)$ for the state of $\xi$ at the space--time
point $(x,t)\in\bZ\times[0,\infty)$. In keeping with standard convention,
diagrams of space--time evolution are drawn with the space coordinate $x$
increasing from left to right, and the time coordinate $t$ increasing from
top to bottom.

A key supporting role will be played by the \add\ CA, a simple
$2$-state rule
denoted by $\lambda_t$, and defined as follows. The states are $0$ and $1$,
and the evolution is
\[
\lambda_{t+1}(x)=\lambda_t(x-1)+\lambda_t(x)+
\lambda_t(x+1) \mmod 2.
\]
As is well known \cite{MOW}, the additive structure of this rule enables many of its
characteristics to be fully understood. (See Figure~\ref{single-schematic}
below for an illustration.)
%
%
\begin{figure}[t]

\includegraphics{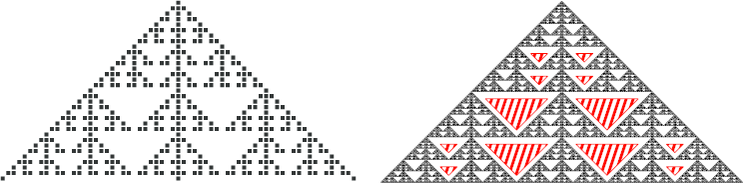}

\caption{Left: the configuration $\lambda^\bullet$ of \textit{1 Or 3},
started from a
single occupied cell, up to time $t=32$.
Right: schematic depiction of a replicator. The striped regions are
filled with a
doubly periodic ether. The thickness of the white ``buffer zones''
remains constant for all time.}
\label{single-schematic}
\end{figure}

%
%
\begin{figure}[t]

\includegraphics{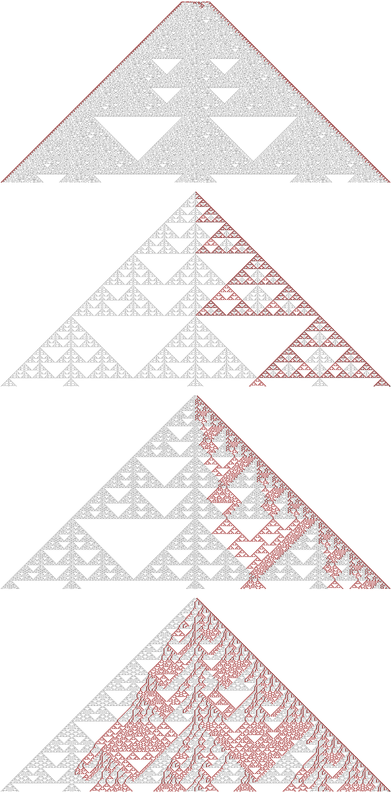}

\caption{Four configurations of \textit{Web-Xor}. The first (top) example,
a replicator with zero ether,
starts from a random string of 64 $1$s and $2$s. The second and third examples,
with respective seeds $12$ and
$11111012$, are quasireplicators. The bottom example, with
seed $1100112$, is apparently chaotic.}
\label{webxor}
\end{figure}

We consider $3$-state CA rules with the following special property. For any
configuration $\xi$, if we define $\lambda_t(x)=\ind[\xi_t(x)=1]$
for all
$x,t$, then $\lambda$ evolves precisely according to the \add\ CA. We also
assume that state $0$ is quiescent, that is, if $\xi_0\equiv0$ then
$\xi_1\equiv0$. We call any CA rule satisfying these two conditions a
\df{web CA}. The idea is that the $1$s form an additive ``web'' which
is not
influenced by the distinction between $0$s and $2$s, while the web may affect
the $2$s. As we will see later, web CA also arise in analysis of
two-dimensional solidification CA. We will usually be interested in
evolution from a \df{seed}, that is, an initial configuration $\xi
_0$ with
finite support.

One of the simplest web CA, which we call \webxor, is defined by setting
$\xi_t(x)=2$ if and only if $\lambda_t(x)=0$ and there is a exactly
one $2$
among $\xi_{t-1}(x-1),\xi_{t-1}(x+1)$. (Together with the web CA condition,
this is sufficient to specify the rule.) Thus, $2$s perform a $2$-neighbor
exclusive-or rule on the points that are not occupied by $1$s.
Figure~\ref{webxor} illustrates the evolution of \webxor\ from four
different seeds. (States are always colored as: $0$ white; $1$ black or
grey; $2$ another color depending on the rule.) Our results imply that
\emph{typical} seeds result in behavior similar to the first picture. More
specifically, we will prove that for certain classes of web CA, evolution
from long random seeds yields with high probability a space--time
configuration that is periodic except within some finite distance of an
additive web. To state this conclusion precisely, we need some more notation.

An \df{ether} is an element $\eta$ of $\{0,2\}^{\bZ^2}$ that is
periodic in
both coordinates. Two ethers are \df{equivalent} if one can be
obtained from
the other via some translation of~$\bZ^2$. In a CA configuration $\xi$,
we say that a
set $K\subseteq\bZ\times[0,\infty)$ is \df{filled with} $\eta$ if
$\xi$
agrees with some ether equivalent to $\eta$ on $K$. Let $\lambda
^\bullet
$ be
the \add\ CA started from the seed consisting of a single $1$ at the origin,
and let $\Lambda=\{(x,t)\dvtx\lambda^\bullet(x,t)=1\}$ be its
support. See
Figure~\ref{single-schematic}. Let $\Lambda(r)\subset\bZ^2$ be the
set of
space--time points at $\ell^1$-distance at most $r$ from $\Lambda$.

%
For a given CA, we say that a seed $\xi_0$ (or equivalently the resulting
configuration~$\xi$) is a \df{replicator} of \df{thickness} $r$ and ether
$\eta$ if each bounded component of $\bZ^2\setminus\Lambda(r)$ is filled
with $\eta$. See Figure~\ref{single-schematic}. It is a straightforward
fact that $\Lambda(r)$ has density $0$ as a subset of $\bZ^2$ for any $r$.
Therefore, in a replicator, the density of $2$s within the cone
$\{(x,t)\dvtx|x|\leq t\}$ equals the density of the ether.
Furthermore, it may
be shown that for any replicator (of any CA), the configuration $\xi$
can be
fully described in terms of a finite set of local patterns that are repeated
at infinitely many locations prescribed by $\Lambda$. (This is the reason
for the name replicator.) For more details, we refer the reader to
\cite{GGP}, where the concept was introduced.

Our results will apply to web CA rules satisfying two conditions which we
call \textit{diagonal-compliance} and \textit{wide-compliance}. The
conditions state
that flow of information concerning the distinction between $0$s and
$2$s is
blocked by certain local patterns of $1$s. The formal statements of the
conditions are straightforward but somewhat technical, and we therefore
postpone them
to the next section. For now, we note that \webxor\ is diagonal-compliant.

A \df{uniformly random binary seed} on $[0,L]$ is an initial configuration
$\xi_0$ in which $\xi_0(x)$ takes values $0,1$ with equal probabilities
independently for all $x\in[0,L]$, and $0$ outside $[0,L]$.

%
\begin{theorem}[(Replication from random seeds)]\label{showcase}
Consider a web CA that is either diagonal-compliant or wide-compliant,
started from a uniformly random binary seed on $[0,L]$. There exist a
random variable $R_L$ taking values in $[0,\infty]$, and a random ether
$\eta_L$ (both deterministic functions of the seed), with the following
properties. We have $\P(R_L=\infty)\to0$ as $L\to\infty$, and
indeed the
sequence $(R_L)_{L\ge0}$ is tight. On the event $R_L<\infty$, the
configuration $\xi$ is a replicator of thickness $R_L+L$ and ether
$\eta_L$.
Furthermore, if any finite set of 0s in $\xi_0$ are changed into 2s,
the same
statement holds with the same $R_L$ and $\eta_L$.
\end{theorem}

Web CA rules may be further classified in the following way, which has
implications for their production of ethers. A CA has \df{no spontaneous
birth} (of 2s) if whenever $\xi_0$ contains no 2s, $\xi_1$ also
contains no
2s. \webxor\ has no spontaneous birth. Figure~\ref{piggy} shows four
possible evolutions of a CA rule called \pig\ (to be defined in the next
section) that is wide-compliant and has spontaneous birth.
%
%
\begin{figure}

\includegraphics{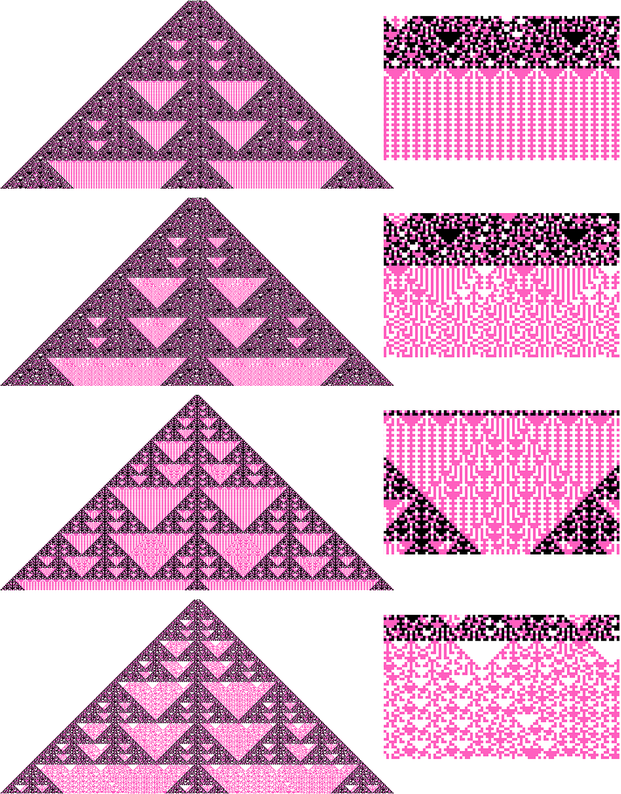}

\caption{Four examples of \textit{Piggyback} evolution: two replicators
(with enlarged
regions showing different ethers) from random seeds of length 30; a
quasireplicator
with seed $11111$; and an apparently chaotic example with seed
$100011011$.}\label{piggy}
\end{figure}

%
%
%
%
%
%
%
%
%
%

%
\begin{theorem}[(Trivial and nontrivial ethers)]\label{ether-births}
Assuming the conditions of Theorem~\ref{showcase}, $R_L$ can be chosen to
have the following additional properties:
%
\begin{longlist}[(ii)]
%
\item[(i)] If the CA rule has no spontaneous birth, then $\eta_L\equiv0$
whenever $R_L<\infty$.
\item[(ii)] Suppose that the CA rule has spontaneous birth. If for some
deterministic ether $\eta$ we have $R_L<\infty$ and $\eta_L=\eta$ for
\emph{some} binary seed, then for uniformly random binary seeds we have
\[
\liminf_{L\to\infty} \P(R_L<\infty\mbox{ and }
\eta_L=\eta)>0.
\]
\end{longlist}
%
\end{theorem}

Given any particular seed, there is a simple procedure to compute the random
variable $R_L$ appearing in Theorems \ref{showcase} and \ref{ether-births},
and in particular to determine whether it is finite. (See Sections \ref{ca}
and \ref{ethers} for details.) For many CA of interest, including \pig, there
are multiple nonequivalent ethers $\eta$ for which the condition of
Theorem~\ref{ether-births}(ii) indeed holds, and which hence have
asymptotically nontrivial probabilities. The first two pictures in
Figure~\ref{piggy} show two examples. Our methods allow the
computation of
explicit rigorous lower bounds on asymptotic probabilities of particular
ethers. For example, in \pig, for the ether that results from the periodic
initial state $(00022222)^\infty$, the $\liminf$ in the theorem is at least
$0.1297$, while $(0)^\infty$, $(2)^\infty$ and $(00002022)^\infty$
have lower
bounds $0.5$, $0.0398$ and $0.0151$, respectively. (In fact, more than
100 ethers
have positive $\liminf$, and we believe that there are infinitely many.)

As remarked earlier, many web CA provably exhibit more complex
behavior for certain exceptional seeds. One important class of behavior is
formalized by the following concept introduced in \cite{GGP}. We call
a seed
$\xi_0$ (or a configuration $\xi$) a \df{quasireplicator} with ether
$\eta$
if the following holds. For some \df{exceptional set} of space--time points
$Q\supseteq\Lambda$, each bounded component of $\bZ^2\setminus Q$ is filled
with $\eta$, while for some $a>1$, the set $a^{-n} Q$ converges as
$n\to\infty$ in Hausdorff metric to a set of Hausdorff dimension strictly
less than~$2$.

%
\begin{theorem}[(Quasireplicators)]\label{quasireplicators}
For some diagonal-compliant and wide-compliant web CA rules, including
\webxor\ and \pig, there exist seeds that are quasireplicators but not
replicators.
\end{theorem}

Examples of (provable) quasireplicators include the second and third
seeds in
Figure~\ref{webxor}, and the third seed in Figure~\ref{piggy}.
Certain other
seeds appear to be neither replicators nor quasireplicators, but exhibit
apparently chaotic behavior, although proving this seems very challenging.
The fourth examples in each figure are in this category. In some very
special cases, we can prove chaotic behavior in a certain conditional sense,
even for an infinite family of seeds whose number grows exponentially with
their length. We discuss these issues further in the next section.

Theorem~\ref{showcase} describes the space--time configuration away from
$\Lambda$, and moreover states that this description is insensitive to $2$s
in the initial configuration. However, the result provides no information
about the configuration close to $\Lambda$. The next result addresses this.
The \df{forward cone} of a space--time point $(x,t)$ is the set
$\{(y,s)\dvtx|y-x|\leq s-t\}$, and the forward cone of a set is the union
of the
forward cones of its points.

%
\begin{theorem}[(Stability)]\label{stability}
Consider a diagonal-compliant or wide-\break compliant web CA, started from a
uniformly random binary seed on $[0,L]$. With probability converging to
1 as
$L\to\infty$, the configuration of $\xi$ in the forward cone of
$[0,L]\times\{\lfloor C\log L\rfloor\}$ is unchanged if any set of 0s in
$\xi_0$ are changed to 2s. Here, $C$ is an absolute constant. If the
CA has
no spontaneous birth, then with probability converging to $1$ the same cone
contains no~$2$s.
\end{theorem}

We next discuss some ideas behind our proofs. Since in a web CA the web of
1s evolves according to \add, it easily follows that all 1s lie in
$\Lambda(L)$. In the situation of Theorem~\ref{showcase}, we will
prove that
immediately above each bounded component of $\bZ^2\setminus\Lambda
(L)$ there
is a strip which blocks information flow. Furthermore, each such strip
contains a spatially periodic configuration of 1s, with the repeating unit
being identical for all strips up to translation. This is a probabilistic
statement, not a deterministic one, and the height of the strip is random.
It will be proved using techniques of percolation theory. In contrast with
classical percolation, the space--time configuration $\lambda$ of \add\ is not
i.i.d., but has long-range dependence. We will make use of the key
percolation result below, which we believe is interesting in its own right.

A \df{path} is a finite or infinite sequence $\pi$ of space--time points
$(x_0,t_0), (x_1, t_1),\break \ldots,  (x_n, t_n) (\ldots)$ with $t_{i+1}=t_i+1$
and $|x_{i+1}-x_i|\leq1$ for all $i$. A path is \df{diagonal} if it
satisfies $|x_{i+1}-x_i|= 1$ for all $i$. Suppose $\lambda_0$ is
given, and
let $\lambda$ be the resulting configuration of \add. We say that a path
$\pi$ is \df{empty} if $\lambda(x,t)=0$ for every $(x,t)$ on $\pi$.
A path
is \df{wide} if it is empty and it makes no diagonal step between two $1$s,
that is, it has no two consecutive points $(x,t),(y,t+1)$ with
$|x-y|=1$ but
$\lambda(y,t)=\lambda(x,t+1)=1$. (As suggested by the terminology,
diagonal-compliance and wide-compliance of web CA refer to information flow
being restricted to paths of the appropriate type.) We now assume that the
initial configuration $\lambda_0$ is uniformly random on $\bZ$, that is,
$\lambda_0(x)$ takes values $0,1$ with equal probabilities
independently for
all $x\in\bZ$.

%
\begin{theorem}[(Subcriticality)]\label{no-percolation}
Consider the \add\ CA from a uniformly random initial configuration on
$\bZ$.
We have
%
%
\begin{equation}
\label{exponential-bound} \P \bigl(\exists\mbox{ an empty diagonal path from $\bZ\times\{0
\} $ to $(0,t)$} \bigr)<e^{-ct},\qquad t>0,\hspace*{-25pt}
\end{equation}
for some absolute constant $c>0$. The same conclusion holds for the
existence of a wide path.
\end{theorem}

In contrast, we prove that empty paths do percolate.

%
\begin{theorem}[(Supercriticality)]\label{percolation}
For the \add\ CA from a uniformly random initial configuration on $\bZ$,
\[
\P \bigl(\exists\mbox{ an infinite empty path from }(0,0) \bigr)>0.
\]
\end{theorem}

We now briefly discuss background to our results. As remarked earlier, CA
that exhibit chaotic behavior for typical seeds but regular behavior for
some seeds are apparently very common. Empirical evidence strongly suggests
that the one-dimensional rules \textit{Exactly 1} \cite{GG3}, \textit
{Perturbed
Exactly 1} \cite{GGP} and \textit{EEED} \cite{GG4} are all in this
category. It is natural to postulate a mechanism for this phenomenon,
whereby chaos nucleates from certain local patterns, and, once started,
invades all nonchaotic regions. It is tempting to conclude that this
robustness of chaos might be universal law, akin to the second law of
thermodynamics.

To our knowledge, the first compelling evidence to the contrary was presented
in~\cite{GG2}, where a CA later called \textit{Extended 1 Or 3} was
introduced.
This rule arises naturally as the ``2-layer extremal boundary
dynamics'' of
a classical two-dimensional CA rule, \textit{Box 13}. \pig\ is also the 2-layer
extremal boundary dynamics of a two-dimensional rule. See
Section~\ref{prelim} for more information. Extremal dynamics have been
utilized very effectively in the analysis of Packard snowflake CA in~\cite{GG1,GG2}.

\textit{Extended 1 Or 3} was proved in \cite{GGP} to admit both
replicators and
quasireplicators, and observed to generate apparent chaos from some seeds.
Empirical evidence was presented that long random seeds are replicators with
high probability, and thus that it is the ordered phase that is
resistant to
disorder. In this article, we provide the first rigorous demonstration of
this phenomenon. The classes of CA that we consider are strongly
inspired by
\textit{Extended 1 Or~3}. We have not succeeded in proving that the conclusion
of Theorem~\ref{showcase} holds in the case of \textit{Extended 1 Or~3},
although this would follow if a certain natural conjecture
(Conjecture \ref{4-free}) were established.

We note that the disorder-resistance phenomenon under consideration is
somewhat reminiscent of insensitivity of CA rules to random noise in the
update rule, as in \cite{Gac1,Gra}.

Much CA research has focused on evolution from carefully chosen initial
configurations---a notable rigorous example is \cite{Coo}. In contrast,
rigorous results for CA from a random initial configurations are scarce,
despite their potential importance in understanding self-organization. Most
such research has been focused on \textit{nucleation}, that is, random
formation
of centers that orchestrate a takeover of the available space. Notable
examples include bootstrap percolation \cite{Hol,BBDM} and excitable media
models \cite{FGG}. We also mention two previous works on additive dynamics
started
from a product measure, \cite{Lin} and \cite{EN}; the latter finds an
embedded random walk by an argument somewhat related to the methods in
Section~\ref{subcritical}.

In many cases, percolation with long range dependence is extremely
challenging to analyze rigorously (see \cite{Win,BBS,Gac2,BS}, and
references therein). Nevertheless, in our setting it turns out that the
additivity of \add\ allows certain judiciously chosen percolation arguments
to be carried through with relative ease. Translating results from an
infinite random initial configuration to finite seeds also appears daunting,
since the number of random bits is now \textit{finite}. However, additivity
introduces extensive periodicity and repetition into the configuration. With
care, these properties can be used to advantage. This extreme form of
long-range dependence provides the link between lack of percolation and
evolution from random seeds, and is also the reason for formation of ethers.

While our results provide a reasonably comprehensive picture of \emph
{subcritical} percolation behavior for certain path types (diagonal and wide),
it should be emphasized that the behavior for paths of \emph{supercitical}
type (empty paths) in the evolution from finite random seeds is not well
understood. We discuss open questions and prove some preliminary
results in
this direction in Section~\ref{supercritical}.

The article is organized as follows. In Section~\ref{prelim}, we establish
terminology, including the formal definitions of diagonal-compliance and
wide-compliance, we introduce and discuss some further examples of CA having
these properties, and we discuss how Theorem~\ref{quasireplicators} is
proved. Sections \ref{blob}--\ref{seeds} are concerned entirely with
properties of the additive rule \add, from which properties of web CA are
deduced later. In Section~\ref{blob}, we review properties (most of
them well
known) of \add\ started from a single occupied site, and in
Section~\ref{dual} we use additivity to deduce basic properties of the
evolution from random configurations. In Sections \ref{subcritical} and
\ref{supercritical}, we prove the percolation results,
Theorems \ref{no-percolation} and \ref{percolation}, respectively,
and discuss
other facts and open problems concerning percolation. In Section~\ref{seeds},
we deduce key results about evolution of \add\ from random seeds.
Finally, we
return to web CA. In Section~\ref{ca}, we deduce Theorems \ref
{showcase} and
\ref{stability}, and in Section~\ref{ethers} we prove
Theorem~\ref{ether-births} and show how to compute lower bounds on ether
probabilities.

\section{Definitions, examples and preliminary results} \label{prelim}

\subsection{Basic conventions}

Throughout the paper, $\lambda$ denotes the \textit{1 Or 3} CA, while
$\xi$
denotes a web CA. All our intervals will be subsets of $\bZ$ or of
$\bZ\times\{t\}$ for some $t\ge0$. We adopt the convention that
$[a,b]=\varnothing$ and $[a,b]\times\{t\}=\varnothing$ whenever $b<a$.

Throughout, a \textit{site} or a \textit{cell} will refer to an
element of
$\bZ$; a \textit{point} will be an element of space--time
$\bZ\times[0,\infty)\subset\bZ^2$. The state of a CA $\xi$ at
cell $x$ and
time $t$ is denoted $\xi_t(x)$ or $\xi(x,t)$, depending on whether
our focus
is on time evolution or the space--time configuration. When specifying
a seed,
we always assume that all states left unspecified are 0. In diagrams of
space--time evolution, state $0$ is colored white, state $1$ is black
or grey
and state $2$ is a different nongreyscale color for each CA rule.

We say that a collection of $\{0,1\}$-valued random variables is
\df{uniformly random} if they are independent and take values $0$ and $1$
each with probability $1/2$.

\subsection{Compliance}

In this section, we formally introduce various families of web CA.
As mentioned already, these will have $3$ states and range $2$.
Thus, the state of a site is
$\xi_t(x)\in\{0,1,2\}$ for $x\in\bZ$ and $t\in[0,\infty)$, and the
evolution is given by
\[
\xi_{t+1}(x)=f \bigl(\xi_{t}(x-2),\xi_{t}(x-1),
\xi_{t}(x), \xi_{t}(x+1),\xi_{t}(x+2) \bigr)
\]
for some function $f$.

We reiterate our standing assumption that the $1$s of $\xi$ behave as
the \textit{1 Or 3} CA. More precisely, writing
\[
\delta(a):=\ind[a=1]=a \mmod2,\qquad a=0,1,2,
\]
we assume
that
%
%
\begin{equation}
\label{first-level} \delta\bigl(f(a,b,c,d,e)\bigr)=\delta(b)+\delta(c)+\delta(d)
\mmod2
\end{equation}
for all $a,b,c,d,e$. Thus, if we define
%
%
\begin{equation}
\label{lambda} \lambda_t(x):=\delta\bigl(\xi_t(x)
\bigr),
\end{equation}
then \eqref{first-level} implies that $\lambda$ satisfies
the \textit{1 Or 3} CA rule. We sometimes call $\lambda$ the
\df{first level} of the process. We call a CA rule that
satisfies \eqref{first-level} and $f(0,0,0,0,0)=0$ a~\df{web} rule.

We now consider various further conditions that may be imposed
on $f$. The idea will be that the flow of information
concerning the distinction between states $0$ and $2$ is
blocked by $1$s (in various locations). Throughout the following, we take
$a,b,c,d,e$ and $a',b',c',d',e'$ to be arbitrary satisfying
$\delta(a)=\delta(a')$, $\delta(b)=\delta(b')$, etc.

We say that the rule $f$ is \df{empty-compliant} if
\[
f(a,b,c,d,e)=f\bigl(a',b,c,d,e'\bigr);
\]
that is, a cell's next state $\xi_{t+1}(x)$ depends on
nonadjacent cells $\xi_t(x\pm2)$ only through their first
level. [Recall that by \eqref{first-level}, the \emph{first
level} of the next state cannot depend on the nonadjacent
cells at all.] Similarly, we say that the rule is \df{diagonal-compliant}
if
\[
f(a,b,c,d,e)=f\bigl(a',b,c',d,e'
\bigr).
\]

It will be convenient to express the next conditions in terms
of the \emph{new} first-level states of the neighboring cells.
Thus, we denote
\begin{eqnarray*}
\ell&:=&\delta(a)+\delta(b)+\delta(c)\mmod2;
\\
r&:=&\delta(c)+\delta(d)+\delta(e)\mmod2,
\end{eqnarray*}
so that if $(a,b,c,d,e)=(\xi_t(x-2),\ldots,\xi_t(x+2))$ then
$(\ell,r)=(\lambda_{t+1}(x-1),\lambda_{t+1}(x+1))$. We say that
$f$ is \df{wide-compliant} if it is empty-compliant and
\begin{eqnarray*}
c&=&r=1 \qquad\mbox{implies } f(a,b,c,d,e)=f\bigl(a',b,c,d',e'
\bigr)\qquad\mbox{and}
\\
c&=&\ell=1 \qquad\mbox{implies } f(a,b,c,d,e)=f\bigl(a',b',c,d,e'
\bigr).
\end{eqnarray*}

In a configuration $\lambda$ of \add, a path is said to be \df
{$\theta$-free}
if it is empty and it contains no point $(x,t)$ whose $5$-point neighborhood
$\{(x\pm1,t),(x\pm1,t-1),(x,t-1)\}$ contains $\theta$ or more $1$s. Finally,
we say a CA rule $f$ is \df{$\theta$-free-compliant} if it is empty-compliant
and
%
\begin{eqnarray}
\delta(b)+\delta(c)+\delta(d)+\delta(\ell)+\delta(r)\geq\theta
\nonumber\\
\eqntext{\mbox{implies } f(a,b,c,d,e)=f\bigl(a',b',c',d',e'
\bigr).}
\end{eqnarray}

Recall the definition of \emph{no spontaneous birth} from the
\hyperref[sec1]{Introduction};
this is equivalent to the condition that $f(a,b,c,d,e)\ne2$
whenever $a,b,c,d,e\in\{0,1\}$.

As suggested by the terminology, the behavior of cellular
automata satisfying the above conditions is constrained by
paths of the appropriate types.

%
\begin{lemma}[(Compliance)]\label{compliance}
Consider a web CA that is empty-compliant
(resp.: diagonal-compliant, wide-compliant, or $\theta$-free-compliant).
Consider two initial configurations $\xi_0,\xi'_0$ whose first
levels agree [i.e., $\delta(\xi_0(x))=\delta(\xi'_0(x))$ for
all $x$], and define the first-level dynamics $\lambda$ via
\eqref{lambda}. Fix a point $(y,t)$. If $\lambda$ has no empty
path (resp.: empty diagonal, wide, or $\theta$-free
path) from any $(x,0)$ at which $\xi_0(x)\neq\xi'_0(x)$ to
$(y,t)$, then $\xi_t(y)=\xi'_t(y)$. Moreover, if the CA has
no spontaneous birth, then $\xi_t(y)\neq2$.
\end{lemma}

\begin{pf}
Suppose, to the contrary, that $\xi_t(y)\ne\xi_t(y')$. We need
to show that there exists a path of the appropriate type
from $\bZ\times\{0\}$ to $(y,t)$. By induction, it suffices to exhibit
the final step on this path. This is a straightforward verification.

To prove the final claim in the no spontaneous birth case, consider the
initial state $\xi_0'$ in which every 2 of $\xi_0$ is changed to 0. Then
$\xi_t(y)=\xi_t'(y)=0$.
\end{pf}

%
\begin{lemma}[($3$-free paths)]\label{3-free}
In any configuration $\lambda$ of \add, any $3$-free path is wide. Any
$3$-free-compliant web CA rule is wide-compliant.
\end{lemma}

\begin{pf}
Assume that a $3$-free path makes a leftward diagonal move
on two space--time points in state 0. Denote the
states $a,b,c$ at nearby points thus:
\[
\matrix{ & a& b &0
\vspace*{2pt}\cr
&&0 &c }
\]
We need to show that $b$ and $c$ cannot be both 1. However, if $b=1$, then
also $a=1$, but then $c=0$ as the path is $3$-free. This establishes the
first claim. A similar argument gives the second claim.
\end{pf}

We now state a simple but important lemma that says that,
although the web rules have range 2, empty-compliance
ensures that the ``light speed'' is essentially~$1$.

%
\begin{lemma}[(Light speed)]\label{light-speed}
Assume an empty-compliant web CA.
The state $\xi(x,t)$ depends on the initial configuration $\xi_0$
only through
the states
\[
\lambda_0(x-t-1), \qquad\xi_0(x-t), \ldots,
\xi_0(x+t), \qquad\lambda_0(x+t+1),
\]
where $\lambda$ is defined by \eqref{lambda}.
\end{lemma}

\begin{pf}
The given states determine the following states at time $1$:
\[
\lambda_1(x-t), \qquad\xi_1(x-t+1), \ldots,
\xi_1(x+t-1),\qquad \lambda_1(x+t).
\]
Then we use induction.
\end{pf}

\subsection{Examples of rules}

We will introduce several examples of web CA, chosen to represent
various behaviors.
Finding such rules is not particularly difficult, and we know of many
others with
similar characteristics.
Let the function $N_1$ (resp., $N_{2}$) count the number of $1$s
(resp., $2$s)
among its arguments, and $N_{12}=N_1+N_2$.

Our first example is \textit{Web-Xor}, whose update rule is given by
\[
f(a,b,c,d,e)= \cases{ 1,&\quad  $(b+c+d)\mmod2=1,$\vspace*{2pt}
\cr
2, &\quad $(b+c+d)\mmod2=0
\mbox{ and $N_2(b,d)=1$},$\vspace*{2pt}
\cr
0, & \quad$\mbox{otherwise}$.}
\]
It is easy to check that \textit{Web-Xor} is diagonal-compliant and
has no
spontaneous birth. Examples of its evolution are given in Figure~\ref
{webxor}. The top example represents typical behavior: replication with
zero ether from a long random seed. The middle two examples are
quasireplicators, one very simple and one similar to the one in
Theorem~8 of
\cite{GGP}. For many seeds including these two, quasireplication can be
rigorously proved via inductive schemes that completely characterize the
configuration at certain specified times. In more complicated cases, such
schemes can be very laborious to construct, while in other cases it may be
difficult even to determine whether the seed is a quasireplicator. We will
not give proofs of quasireplication; instead we refer the reader to
\cite{GGP} for two typical examples of inductive schemes that feature
in such
arguments. We believe that the final example in Figure~\ref{webxor} is
chaotic.

Even this simplest of rules displays a remarkable variety of behavior from
``exceptional'' seeds. Other interesting seeds that we have found include:
110010012 (a replicator with nontrivial pattern of $2$s in the web), 110011112
(a quasireplicator with scale factor $a=4$), 111001112 (perhaps chaotic
or a very
complicated quasireplicator), 10110112 (apparent chaos restricted to one
side).

\textit{Modified Web-Xor} also has no spontaneous birth,
but the $2$s obey a symmetric two-point \textit{Or} rule in the presence
of $1$s:
\[
f(a,b,c,d,e)= \cases{1, &\quad  $(b+c+d)\mmod2=1,$\vspace*{2pt}
\cr
2, &\quad  $(b+c+d)
\mmod2=0, \mbox{ and}$\vspace*{2pt}
\cr
& \quad$\mbox{either $N_2(b,d)=1$ or}$\vspace*{2pt}\cr
&\quad$\mbox{$\bigl[N_2(b,d)> 1$ and $N_1(\ell, b, c, d, r)\ge1\bigr]$},$\vspace*{2pt}
\cr
0, & \quad$\mbox{otherwise}$.}
\]
As seen in Figure~\ref{mod-webxor}, this rule is capable of ``mixed
replication'' with two different ethers (top). Note that
Theorem~\ref{showcase} implies that with high probability this does not
happen for long random seeds. The bottom example is apparently a
quasireplicator, although we have no proof, and it seems that the inductive
methods of \cite{GGP} do not apply. Here and in the last example of
Figure~\ref{webxor}, it is plausible that the evolution is driven by the
advance of a front that lags behind the edge of the light cone by a power
law. We will discuss this phenomenon in Section~\ref{supercritical}.
%
%
\begin{figure}

\includegraphics{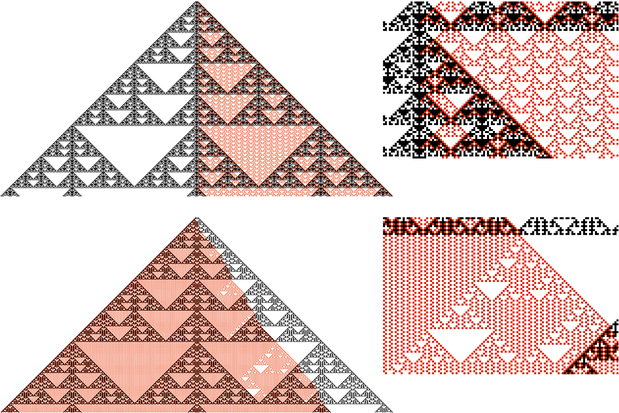}

\caption{\textit{Modified Web-Xor} with seeds $11111112$ and
$210001$.}\label{mod-webxor}
\end{figure}

%
%
%
%
%

In \textit{Web-adapted Rule 30}, $2$s evolve according to \textit
{Rule 30}
\cite{Wol2}, except that $2$s perform the three-point \textit{Or} rule
in the
presence of $1$s when a neighborhood occupation number is small enough:
\[
f(a,b,c,d,e)= \cases{ 1, & \quad$(b+c+d)\mmod2=1,$\vspace*{2pt}
\cr
2, & \quad $(b+c+d)\mmod2=0
\mbox{ and }N_{1}(\ell, b, c, d, r)\le2, \mbox{ and}$\vspace*{2pt}
\cr
&\quad
$\mbox{either $w_{30}\bigl[\delta_2(b),
\delta_2(c),\delta _2(d)\bigr]=1$}$\vspace*{2pt}
\cr
&\quad $\mbox{or $\bigl[N_2(b,c,d)\ge1$ and $N_1(
\ell, b, c, d, r)\ge 1\bigr]$},$\vspace*{2pt}
\cr
0, &\quad $\mbox{otherwise}$. }
\]
Here, $w_{30}$ is the update rule for \textit{Rule 30}, given by
$w_{30}(a_1,a_2,a_3)=(a_1+a_2+a_3+a_2a_3)\mmod2$, and
$\delta_2(a):=\ind[a=2]$. \textit{Web-adapted Rule 30} is $3$-free-compliant
(and therefore wide-compliant) and has no spontaneous birth. See
Figure~\ref{rule30} for an example. One can prove that this instance is
not a
replicator, but is it chaotic? There are no known methods to prove chaotic
evolution, or even universally agreed definitions of the concept; however,
suppose one accepts the reasonable premise that \textit{Rule 30}
generates a
chaotic configuration $\rho$ started from a single $1$ \cite{Wol2}.
Then the
example in Figure~\ref{rule30} is equally chaotic, in the sense that its
evolution provably features larger and larger regions of $\rho$, at specific
locations that are easily characterized. We will also show in
Section~\ref{subcritical} that an exponentially growing family of seeds
exhibit conditional chaos in the same sense.
%
%
\begin{figure}

\includegraphics{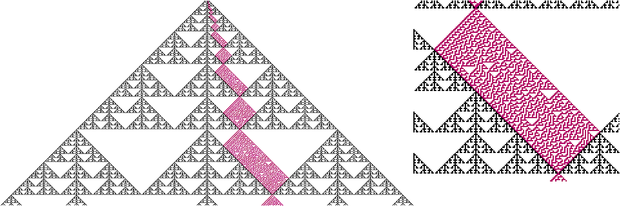}

\caption{Chaotic behavior of \textit{Web-adapted Rule 30} with seed
$100010201$.}
\label{rule30}
\end{figure}
%
%
%
%

The above rule may be modified in various ways so as to include spontaneous
birth, resulting in further rules where Theorem~\ref{showcase}
applies, yet
in which many provable replicators have ethers with very long temporal
period, perhaps too long to be seen experimentally. In the interest of
brevity, we omit the details. We briefly discuss bounds on the period in
Section~\ref{ca}.

A number of web rules arise naturally in analysis of two-dimensional
CA, as
we now explain. Consider a binary CA $\zeta_t\in\{0,1\}^{\bZ^2}$, in which
the new state of cell $z$ is given by a rule defined on the Moore
neighborhood $\cN(z):=\{z'\in\bZ^2\dvtx\Vert z'-z\Vert_\infty\le
1\}$. We
assume that
state $0$ is quiescent, and that the CA \df{solidifies}, that is,
$\zeta_t(z)=1$ implies $\zeta_{t+1}(z)=1$; the CA rule then only
needs to
specify when a $z\in\bZ^2$ becomes occupied, that is, changes its state
from 0
at time $t$ to 1 at time $t+1$. To each such CA, we associate \df{extremal
boundary dynamics (EBD)}: assume that $\zeta_0$ vanishes on $\bZ
\times[1,
\infty)$ and let $\lambda_t$ be given by $\zeta_t$ on $\bZ\times\{
t\}$.
Observe that $\lambda_t$ is a one-dimensional CA whose space--time
configuration is a lower bound on the \df{final configuration}
$\zeta_\infty=\bigcup_{t\ge0}\zeta_t$. Now assume that we extend
the boundary
layer to width 2, which leads to the CA $\xi_t\in\{0,1,2\}^{\bZ}$
with the
following rule: $\xi_t(x)=1$ if $\zeta_t(x,t)=1$ (so that
$\lambda_t=\xi_t\mmod2$), $\xi_t(x)=2$ if $\zeta_t(x,t)=0$ but
$\zeta_{t+1}(x,t)=1$, and $\xi_t(x)=0$ otherwise. Again, $\xi_t$ is a
one-dimensional CA. As $\zeta_t(x,t-1)=1$ exactly when either $\xi
_{t-1}(x)=1$ or
$\xi_t(x)=2$, $\xi_t$ indeed determines \textit{two} extremal layers of
$\zeta_t$, and is thus called \df{two-level EBD}. The evolution of
$\xi_t$
also provides a lower bound on $\zeta_\infty$ and is often useful
when the
bound provided by $\lambda_t$ ``leaks''~\cite{GG2}. To conform with
the rest
of the paper, we assume throughout that the EBD is the \textit{1 Or 3} CA.

The natural setting for study of the issues addressed in this paper are
general web CA, a much larger class than the two-level EBD rules. The latter,
however, provide many interesting examples. In fact, the different ethers,
quasireplicators and (apparent) chaotic behavior were first observed in the
two-level EBD generated by the \textit{Box 13} solidification CA \cite{GG2},
in which $z$ becomes occupied at time $t+1$ when the number of occupied
cells in $\cN(z)$ at time $t$ is 1 or 3. The corresponding two-level
EBD is
called the \textit{Extended 1 Or 3} CA, and is given by
\[
f(a,b,c,d,e)= \cases{1, &\quad $(b+c+d)\mmod2=1,$\vspace*{2pt}
\cr
2, &\quad $(b+c+d)\mmod2=0
\mbox{ and } N_{12}(\ell, r, b, c, d)\in\{ 1,3\}$,\vspace*{2pt}
\cr
0, &\quad $
\mbox{otherwise}$,}
\]
as is easy to check; therefore, this rule is equivalent to the one with the
same name introduced in \cite{GGP}. This rule is $4$-free-compliant,
and is
not covered by our main theorems. However, we establish some rigorous
results in Section~\ref{ethers}.

%
%
\begin{figure}[b]

\includegraphics{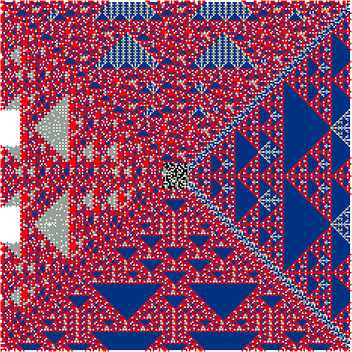}

\caption{\pbox\ started from a seed in the square $[0,16]^2$.
Initially occupied cells are black, and subsequently occupied cells
are red or blue if they have state $1$ or $2$, respectively, in the
$2$-level EBD, and otherwise grey. Unoccupied cells are white.}
\label{box-piggyback}
\end{figure}

\newcommand{\no}{\mathtt{occ}}
For simplicity, assume that the two-dimensional CA $\zeta$ is \df{isotropic},
that is, that its rule respects all isometries of the lattice $\bZ^2$.
Then there is a convenient sufficient condition that assures wide-compliance
for its two-level EBD:
when the neighborhood configuration is
\[
\matrix{ a & 1 & c\vspace*{2pt}
\cr
b & 0 & 1\vspace*{2pt}
\cr
0 & 0 &
0 } %
\]
the next state at the center cell is independent of $c$ (i.e., depends only
on $a$ and~$b$). This holds, for example, for the following solidification
rule, which we call \pbox. Given
$\zeta_t$, let $\no_1(z)$ [resp., $\no_\infty(z)$] count the number of
occupied cells among the four nearest neighbors of $z$ [resp., in $\cN(z)$];
then $z$ becomes occupied if either
\begin{longlist}
%
\item[$\bullet$] $\no_1(z)=2$, or
\item[$\bullet$] $\no_1(z)\le1$ and $\no_\infty(z)\in\{1,3\}$.
\end{longlist}
See Figure~\ref{box-piggyback} for an example.\vadjust{\goodbreak}

The resulting two-level EBD has the update rule
\[
f(a,b,c,d,e)= \cases{1, &\quad $ (b+c+d)\mmod2=1,$\vspace*{2pt}
\cr
2, & \quad $(b+c+d)
\mmod2=0,\mbox{ and}$\vspace*{2pt}
\cr
&\quad $\mbox{either }N_{12}(
\ell,c,r)=2 $\vspace*{2pt}
\cr
& \quad$\mbox{or } \bigl[N_{12}(\ell,c,r)\le1
\mbox{ and }N_{12}(\ell, b, c, d,r)\in\{ 1,3\} \bigr],$\vspace*{2pt}
\cr
0,
& \quad $\mbox{otherwise}$.}
\]
We call this web CA \textit{Piggyback}. It is easy to see that it is
wide-compliant, and has spontaneous birth. The top two examples in
Figure~\ref{piggy} start from long random seeds and are replicators with
different ethers. (We will have more to say about ethers for \textit{Piggyback}
in Section~\ref{ethers}.) The third example is provably
nonreplicating, as it is a quasireplicator. The bottom example appears
to be
chaotic. Like the bottom picture in Figure~\ref{webxor}, the evolution
displays a tantalizing mixture of order and disorder.

Our results on \pig\ have rigorous implications for the two-dimen\-sional
\pbox\ rule (and similarly in other cases where 2-level EBD satisfies the
conditions of Theorem~\ref{showcase}). Here, we summarize some initial
observations, noting that further investigation is warranted. As suggested
by Figure~\ref{box-piggyback}, the evolution of \pbox\ from a seed in
$[0,L]^2$ is governed by four space--time configurations of \pig\ in four
quadrants with boundaries at $45^\circ$ to the axes. Depending on the
behavior of each, we may make deductions about the final configuration
$\zeta_\infty$. In the case of a replicator with the ``solid'' ether
$(2)^\infty$, as in the bottom quadrant in this example, clearly no further
filling of the ether is possible after the second level of the EBD. By
Theorem~\ref{ether-births}, it follows that \pbox\ started from a
uniformly random seed in $[0,L]^2$ results in a final configuration
$\zeta_\infty$ of density $1$ in $\bZ^2$ with probability bounded
away from
$0$ as $L\to\infty$. Certain other ethers of \pig\ can also be shown
to fill in
in a predictable manner, resulting in a corresponding ether for
\pbox,
as in the top quadrant. A~similar analysis can likely be carried
through for
certain simple quasireplicators such as the one in the right quadrant. When
\pig\ is a replicator with zero ether, as in the left quadrant, it appears
plausible that the subsequent filling-in by \pbox\ results in a
chaotic final
configuration. See \cite{GG1,GG2} for detailed analysis of the filling-in
process for some other EBD.

We conclude by mentioning a natural rule that seems intractable
by our current methods. \textit{Web 1 Or 3} is the web CA in which 2s perform
\textit{1 Or 3} on the points not occupied by 1s:
\[
f(a,b,c,d,e)= \cases{ 1, &\quad $(b+c+d)\mmod2=1,$\vspace*{2pt}
\cr
2, &\quad $(b+c+d)\mmod2=0
\mbox{ and $N_2(b,c,d)\mmod2=1$},$\vspace*{2pt}
\cr
0, &\quad $
\mbox{otherwise}$.}
\]
Figure~\ref{web1or3} gives an example of an evolution from a random seed
of $1$s and $2$s, with an apparent message of near-criticality
and chaos.
%
%
\begin{figure}

\includegraphics{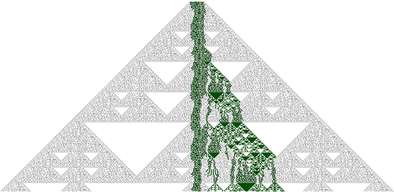}

\caption{Chaotic behavior of \textit{Web 1 Or 3} from a random seed
of 32 sites.}
\label{web1or3}
\end{figure}

\subsection{Generalizations}

The simplest additive rule, \textit{Xor} CA $\mu_t$, is defined on the
state space $\{0,1\}^{\bZ}$ by
\[
\mu_t(x)=\mu_{t-1}(x-1)+\mu_{t-1}(x+1)\mmod2.
\]
One might consider $\mu$, and not $\lambda$, to be the most natural candidate
for the web dynamics. However, while $\mu$ does have some points of interest
(see, e.g., Proposition~\ref{power-law}), many of the main issues we
consider become trivial in this setting. For example, $\mu$ either only
occupies points satisfying a parity constraint or generates an impenetrable
web even for empty paths \cite{BDR,GG1}.

In the other direction, one might ask whether similar results hold if
$\lambda$ is replaced by an arbitrary additive rule. It is indeed likely
that a more general theory could be developed in this setting. One
complication is that predecessors of the all-0 state will no longer
necessarily be unique (as they are for \textit{1 Or 3}---see
Lemma~\ref{predecessor-of-zeros}) and as a result ``mixed replicators''
similar to the top example of Figure~\ref{mod-webxor} may be the norm.

On the other hand, all our results generalize with appropriate minor changes
in the definitions to CA with a quiescent state $0$, first-level state $1$
and other states $2,\ldots, s$.

\section{Additive dynamics from a single occupied site} \label{blob}

Recall that $\lambda^\bullet$ denotes \add\ started from a single $1$.
In this
section, we collect properties that we will need. All these results are
elementary and many are well known. First is a rescaling property,
illustrated in Figure~\ref{rescaling-figure}.
%
%
\begin{figure}

\includegraphics{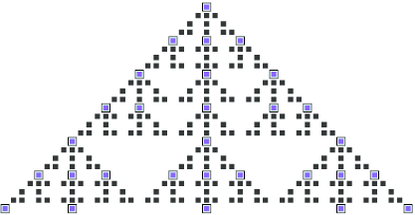}

\caption{An illustration of Lemma \protect\ref{rescaling}, with $m=2$.
Highlighted
points comprise a ``separated out'' copy of $\lambda^\bullet$.}
\label{rescaling-figure}
\end{figure}

%
\begin{lemma}[(Rescaling)]\label{rescaling}
For any nonnegative integers $a$ and $m$,
\[
\lambda^\bullet_{a2^m}(x)= \cases{\lambda^\bullet_a(y),
&\quad $\mbox{if }x=2^my\mbox{ for }y\in\bZ, $\vspace*{2pt}
\cr
0,&\quad $\mbox{otherwise.}$}
\]
\end{lemma}

\begin{pf}
The case $m=1$ follows from additivity on observing that $\lambda
^\bullet_2$
is $10101$. For $m>1$, we apply the $m=1$ case iteratively.
\end{pf}

%
\begin{lemma}[(Periodicity properties)] \label{boundary-periodicity}
%
\begin{longlist}[(iii)]
%
\item[(i)] For $t\ge0$, $\lambda^\bullet_t(0)=\lambda^\bullet_t(\pm
t)=1$ while
$\lambda^\bullet_t(\pm(t-1))=t\mmod2$.
\item[(ii)] For all $n\ge0$, $\lambda^\bullet_{2^n}(x)=1$ exactly at $x=0,
\pm2^n$.
\item[(iii)] For all $n\ge0$, $\lambda^\bullet_{2^n+2^{n-1}}(x)=1$ exactly at
$x=0, \pm2^n, \pm(2^n+2^{n-1})$.
\item[(iv)] For any $k\ge1$, the sequence of edge configurations of
$\lambda^\bullet$ on $[t-k+1,t]\times\{t\}$ is periodic (from $t=0$ on)
with period equal to $2^p$ where $2^{p-1}<k\le2^p$.
\end{longlist}
\end{lemma}

\begin{pf}
Parts (ii) and (iii) follow from Lemma~\ref{rescaling}, and (iv) follows
from (ii), with (i) as a special case. (See Figure \ref{boundary-periodicity-figure}.)
\end{pf}

%
\begin{figure}[b]

\includegraphics{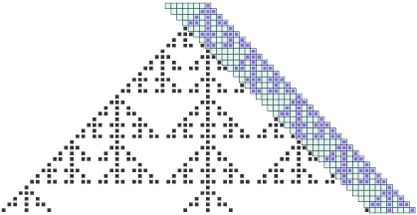}

\caption{Evolution of $\lambda^\bullet$ with
highlighted boundary strip of width $8$ and temporal period $8$.}
\label{boundary-periodicity-figure}
\end{figure}

For some purposes, the following recursive description of
$\lambda^\bullet
$ is
useful, a variant of the one given \cite{Wil}. See
Figure~\ref{B-recursion-figure} for an illustration. Given a space--time
configuration $A$ on $S_n=[0, 2^n]\times[0, 2^n-1]$, we say that $A$ is
\df{placed} at a space--time point $s$ if the configuration in $s+S_n$
is the
corresponding translate of $A$. Let $B_n$ be the space--time
configuration of
$\lambda^\bullet$ on $S_n$. Reflect $B_n$ around its vertical
bisector and
denote the resulting configuration on $S_n$ by $\mB_n$.

%
\begin{lemma}[(Recursion)] \label{B-recursion}
We have $B_0=10$ and
$B_1=
{1\enskip 0\enskip 0\atop 1\enskip 1\enskip 0
}
$. Moreover, for $n\ge2$,
$B_n$ is obtained by placing $B_{n-1}$ at $(0,0)$ and at $(2^{n-1}, 2^{n-1})$;
$B_{n-2}$ at $(0, 2^{n-1})$ and at $(0, 2^{n-1} +2^{n-2})$; and
$\mB_{n-2}$ at $(2^{n-2},2^{n-1})$ and at $(2^{n-2}, 2^{n-1} +2^{n-2})$.
All placements result in consistent state assignments at overlaps.
\end{lemma}

Our results for seeds depend on the fact that $\lambda^\bullet$ has
certain a unique periodic configuration above every region of $0$s.
This property does not hold for general additive rules.

%
\begin{lemma}[(Predecessors of $0$)] \label{predecessor-of-zeros}
For an arbitrary initial state $\lambda_0$, suppose that
$\lambda_t\equiv0$ on $[a,b]$, but $\lambda_{t-1}\not\equiv0$
on $[a-1,b+1]$. Then $\lambda_{t-1}$ is
a subword of the periodic word $(110)^\infty$ on $[a-1,b+1]$.
\end{lemma}

\begin{pf}
Consider the four possible values for the pair $\lambda_{t-1}(a-1)$ and
$\lambda_{t-1}(a)$. Once these states are fixed, the rest of $\lambda_{t-1}$
on $[a-1,b+1]$ can be determined sequentially.
\end{pf}

Fix an initial state for $\lambda$. A \df{void} is a finite inverted triangle
of the form $\bigcup_{i\ge0}([a+i, b-i]\times\{t+i\})$, on which the
configuration is identically 0, and that is maximal with these properties
with respect to inclusion. Its \df{width} is $b-a+1$, and its \df
{start time}
is $t$.

%
\begin{lemma}[(Voids)] \label{voids}
In $\lambda^\bullet$, each void has width $2^k-1$ and start time
divisible by
$2^{k-1}$ for some integer $k\ge1$. Furthermore, for every fixed $k$,
the union of all voids of width at least $2^k-1$
has density 1 within the forward cone of $(0,0)$.
\end{lemma}

\begin{pf}
This is a straightforward application of Lemma~\ref{B-recursion}.
\end{pf}

Finally, we deduce the following fact, which will be crucial for our results
on percolation and ethers. See Figure~\ref{above-voids-figure} for an
illustration.

%
\begin{figure}

\includegraphics{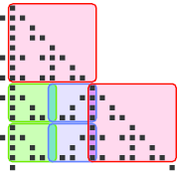}

\caption{Recursive description of $\lambda^\bullet$: $B_4$ is composed
of two copies
of $B_3$ (red), 2 copies of $B_2$ (green) and two copies of $\mB_2$ (blue).}
\label{B-recursion-figure}
\end{figure}

\begin{pf}
This follows easily from (i), (ii) and (iii) of Lemma~\ref
{boundary-periodicity}.
\end{pf}

%
\begin{figure}

\includegraphics{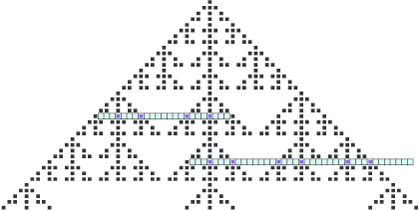}

\caption{Illustration of Proposition \protect\ref{above-voids} with $m=2$.
The highlighted intervals at distance $2^2$ above two selected voids
have the claimed periodic configuration.}
\label{above-voids-figure}
\end{figure}

%
\begin{prop}[(Periodic interval above a void)] \label{above-voids}
In $\lambda^\bullet$, assume that $[a,b]\times\{t\}$ is the top row
of a void of width $2^k-1$. For $m<k$,
the state of interval $[a-2^m,b+2^m]\times\{t-2^m\}$
is a segment of the following infinite periodic string of period
$3\cdot2^m$:
%
%
\begin{equation}
\label{above-voids-string} (1 \lsquare 1 \lsquare 0 \lsquare)^\infty.
\end{equation}
Here, $\lsquare$ represents a string of $2^m-1$ consecutive $0$s, and the
segment begins and ends with a full $\lsquare$.
\end{prop}

\begin{pf}
As $t$ is divisible by $2^{k-1}$, and therefore by $2^m$,
$\lambda^\bullet_t$
on $[a,b]\times\{t\}$ is of the form
\[
\lsquare0 \lsquare\cdots0 \lsquare,
\]
by Lemma~\ref{rescaling}. Then, by the same lemma,
and Lemma~\ref{predecessor-of-zeros} applied to $\lambda_{t/2^m}$,
the configuration on $[a-2^m,b+2^m]\times\{t-2^m\}$ is
either of the claimed type started and ended with $\lsquare$,
or all $0$s. The latter possibility contradicts maximality of the
original void.
\end{pf}

\section{Duality and randomness}\label{dual}

When the initial configuration of \add\ is uniformly random (on some set),
the resulting space--time configuration is of course not uniformly
random but
has a high degree of dependence. Nevertheless, in this section we show how
to identify space--time \emph{sets} on which the randomness is
uniform. The
additive structure of the CA rule ensures that the space--time configuration
is a linear function (modulo $2$) of the initial states, and the idea
is to
find cases where the associated matrix is upper triangular.

Recall that $\lambda^\bullet_t$ is the \textit{1 Or 3} rule started
with only
the origin occupied. Let $\lambda_t^A$ denote the rule started with
the set
of initially occupied sites exactly equal to $A\subseteq\bZ$. We will
extensively use the following version of cancellative duality.

%
\begin{lemma}[(Duality)]\label{duality}
We have $\lambda_t^A(x)=\sum_{y\in A} \lambda^\bullet_t(x-y) \mmod2$.
\end{lemma}

\begin{pf}
This follows easily by additivity and induction on $t$.
\end{pf}

Observe that by symmetry and translation-invariance,
$\lambda^\bullet_t(x-y)=\lambda^\bullet_t(y-x)=\lambda^{\{x\}}_t(y)$.

Suppose we have an ordered set $S=\{(x_i,t_i)\dvtx i=1,2,\ldots, n\}$, of
space--time points. A function $F\dvtx S\to\bZ$ is a \df{dual assignment}
for $S$
if for all $i,j\in\{1,\ldots,n\}$,
\[
\lambda^\bullet \bigl(x_j-F(x_i,t_i),t_j
\bigr)= \cases{ 1,&\quad$\mbox{if }j=i$,\vspace*{2pt}
\cr
0,&\quad $\mbox{if }j<i$. }
\]
(There is no restriction when $j>i$.) We think of $F(\cdot,\cdot)$ as sites
in the initial configuration. The idea is that in order to determine
$\lambda^A(x_i, t_i)$, we need new information about $A$ at each successive
$i$.

%
\begin{prop}[(Randomness via dual assignment)]\label{dual-assignment}
Suppose that the initial configuration $\lambda_0$ of \add\ is uniformly
random on some fixed set $K\subseteq\bZ$ and deterministic on $K^C$.
Let $S$
be a fixed set of space--time points. If $S$ has a dual assignment
whose image
is contained in $K$, then $\lambda$ is uniformly random on $S$.
\end{prop}

\begin{pf}
Writing
\begin{eqnarray*}
K_i&=&\bigl\{y\in K\dvtx\lambda_{t_i}(x_i-y)=1
\bigr\},
\\
K_i'&=&\bigl\{y\in K^C\dvtx
\lambda_{t_i}(x_i-y)=1\bigr\}
\end{eqnarray*}
and
\[
c_i=\sum_{y\in K_i'}\lambda_0(y)
\mmod2,
\]
we have by Lemma~\ref{duality},
\[
\lambda(x_i,t_i)=\sum_{y\in K_i}
\lambda_0(y)+c_i\mmod2.
\]
But $K_i$ contains an element, $F(x_i,t_i)$, that is not in $\bigcup_{j<i}
K_j$, therefore, $\lambda(x_i,t_i)$ is uniformly random conditional on
$(\lambda(x_j,t_j)\dvtx j<i)$.
\end{pf}

A particularly useful special case is that a $1$ adjacent to a string
of $0$s
in $\lambda^\bullet$ heralds uniformly random intervals in the evolution
from a
random seed.

%
\begin{cor}[(Random intervals)]
\label{one-zeros} Fix integers $a$ and $L,k>0$. Let the initial
configuration $\lambda_0$ be a uniformly random binary seed on
$[0,L]$, and
suppose that $\lambda^\bullet_t$ on $[a,a+k]$ is $1$ followed by $k$
$0$s. Then
for any $x\in[-a,L-a-k]$, the configuration $\lambda_t$ is uniformly random
on $[x,x+k]$.
\end{cor}

\begin{pf} By symmetry, $\lambda_t^\bullet$ on $[-a-k, -a]$
is $k$ $0$s followed by $1$. To find a dual assignment of
$[x,x+k]\times
\{t\}$,
order the set from left to right, and let $F(y,t)=y+a$. Clearly, the
image of
this assignment is contained in $[0,L]$. Now apply
Proposition~\ref{dual-assignment}.
\end{pf}

\section{Subcritical percolation}\label{subcritical}

In this section, we prove Theorem~\ref{no-percolation}, which states
that when
\add\ is started from a uniformly random initial configuration on $\bZ
$, the
probability of an empty diagonal or wide path from the initial interval
$\bZ\times\{0\}$ to the point $(0,t)$ decays exponentially in $t$. See
Figures \ref{fluiddiag} and \ref{fluidwide} for the diagonal and wide cases,
respectively. Note the contrast with Figure~\ref{fluidempty} in the next
section for empty paths.

%
\begin{figure}

\includegraphics{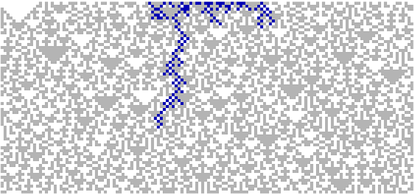}

\caption{All empty diagonal paths from an interval at time $0$ are highlighted
in blue. The initial configuration is uniformly random.}\vspace*{-3pt}
\label{fluiddiag}
\end{figure}

%
\begin{figure}[b]\vspace*{-3pt}

\includegraphics{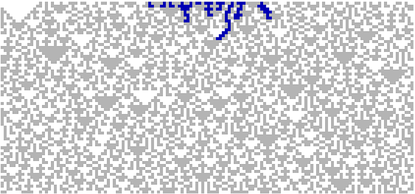}

\caption{All wide paths from an interval at time $0$ are highlighted
in blue.
The initial configuration is uniformly random.}
\label{fluidwide}
\end{figure}

Our approach is to use dual assignments to control the probabilities of
paths, but the details of the argument are very different for the two types
of path. A~diagonal path has $2$ choices at each step, and any given point
has state $0$ with probability $1/2$, suggesting a critical bound. To
improve this to a subcritical bound, we consider a \emph{leftmost}
path, and
use special properties of $\lambda$. On the other hand, we control wide
paths via a random process of space--time intervals that terminates
when an
interval has even length.

Later in the section we also discuss $\theta$-free paths, and show that
notwithstanding Theorem~\ref{no-percolation}, there is an exponential family
of initial configurations for which percolation by wide paths does occur.

\subsection{Empty diagonal paths}
\mbox{}
\begin{pf*}{Proof of Theorem~\ref{no-percolation}, case of empty
diagonal paths}
We may assume without loss of generality that $t$ is even, since the
configuration at time $1$ is also uniformly random, and the probability in
question is strictly less than $1$ for $t=1$.

Fix a diagonal path $\pi$ from $(x,0)$ to $(0,t)$. We will find an upper
bound for the probability that $\pi$ is the \emph{leftmost} empty diagonal
path from $\bZ\times\{0\}$ to $(0,t)$. To this end, partition the
steps of
$\pi$ into segments of length 2. During each such segment, the path
has one
of the following forms: left--left, right--right, left--right or right--left.
When $\pi$ makes a right-left move, that is $(x,s)\to(x+1,s+1)\to(x,s+2)$,
the leftmost property requires a $1$ at $(x-1, s+1)$; we call these points
(which are not on the path) the \df{corner points} of the path, and let
$N(\pi)$ be their number, that is, the number of right--left segments
that start
at even times.

%
\begin{figure}[b]

\includegraphics{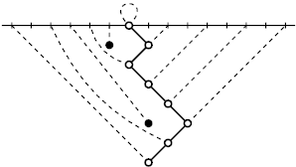}

\caption{A leftmost diagonal path together with a dual assignment for the
points of the path (white discs) and the corner points (black discs). The
dashed lines connect each point to its assigned position in the initial
configuration.} \label{diagonal-figure}
\end{figure}

We will give a dual assignment $F$ of the path together with the set of its
corner points (see Figure~\ref{diagonal-figure} for an illustration). Order
points on the path with increasing time, and place a corner point
$(x,s)$ in
the ordering immediately \textit{after} the point on the path at the time
$s+1$. For every corner point $(x,s)$, let $F(x,s)=x-s+1$. For every point
$(x,s)$ on the path, let $F(x,s)$ be either $x-s$ or $x+s$, according to
whether the path arrives to $(x,s)$ from the right [i.e., from $(x+1,s-1)$]
or from the left [i.e., from $(x-1,s-1)$], respectively. We let $F(x,0)=x$.

To check that $F$ is a dual assignment, we will use
Lemma~\ref{boundary-periodicity}(i). Fix a point $(x,s)$ on the path. All
positions assigned by $F$ to points earlier in the order lie in $[x-s+2,x+s]$
or $[x-s,x+s-2]$ according to whether the path arrives to $(x,s)$ from the
right of left, so the required condition is satisfied for this point. Now
suppose $(x-1, s+1)$ is a corner point arising from the moves
$(x,s)\to(x+1,s+1)\to(x,s+2)$ in the path. This corner point is
assigned to
$F(x-1,s+1)=x-s-1$. We have $F(x,s+2)=x-s-2$, and all points earlier than
$(x,s+2)$ were assigned integers at least $x-s$. Since $s+1$ is odd,
$\lambda^\bullet(x-1-(x-s-1),s+1)=\lambda^\bullet(s,s+1)=1$.
Finally, since $s+2$ is
even, $\lambda^\bullet(x-(x-s-1),s+2)=\lambda^\bullet(s+1,s+2)=0$,
as required.

Now, using Proposition~\ref{dual-assignment},
\begin{eqnarray*}
%
&&\P\bigl(\mbox{an empty diagonal path from $\bZ\times\{0
\}$ to $(0,t)$ exists}\bigr)
\\
&&\qquad\le\sum_\pi\P\bigl(\pi\mbox{ is the leftmost
empty diagonal path from $\bZ \times\{0\}$ to $(0,t)$}\bigr)
\\
&&\qquad\le\sum_\pi\biggl(\frac12\biggr)^{t+1+N(\pi)},
\end{eqnarray*}
where both sums are over all diagonal paths from $\bZ\times\{0\}$ to $(0,t)$.
Let $P_t$ be the last sum above. Then, by considering the last two
steps of
the path,
\[
P_{t+2}= \bigl(\bigl(\tfrac12\bigr)^2+\bigl(\tfrac12\bigr)^2+\bigl(
\tfrac12\bigr)^2+\bigl(\tfrac 12\bigr)^3 \bigr)P_t,
\]
so, recalling that $t$ is even, $P_t=(1/2)\cdot(7/8)^{t/2}$.
\end{pf*}

As an aside, we mention that the assertion of Theorem~\ref{no-percolation}
for diagonal paths also holds when $\lambda$ is replaced by the
\textit{Xor}
CA $\mu$, with a much simpler proof, since the set of all space--time points
that the origin is connected to by diagonal paths is a rectangle.

\subsection{Wide paths}

\mbox{}
\begin{pf*}{Proof of Theorem~\ref{no-percolation}, case of wide paths}
We will prove that
%
%
\begin{equation}
\label{wide-eq1} \P \bigl(\exists\mbox{ a wide path from $(0,0)$ to $\bZ\times\{t
\}$} | \lambda_0(0)=0 \bigr) < e^{-ct}
\end{equation}
for some $c>0$. This clearly suffices by translation-invariance, since there
are only $2t+1$ points at time $0$ from which a path can reach $(0,t)$.
Therefore, we will henceforth assume that $\lambda_0(0)=0$ and that
$\lambda_0$ is uniformly random elsewhere.

We recursively define intervals $I_t=[L_t,R_t]$ for $t=-1, 0,\ldots, T$,
where $T\le\infty$, as follows. Start with $L_{-1}=R_{-1}=0$; then
let $I_0$
be the maximal subinterval of $\bZ$ containing 0 on which $\lambda
_0\equiv
0$. If $I_t=\varnothing$, then we set $T=t$ and there is no $I_{t+1}$.
Otherwise, if $|I_t|\ge2$, then $I_{t+1}$ is the interval $[L_t+1,R_t-1]$
(which is $\varnothing$ when $|I_t|=2$). If $|I_t|=1$, then $I_{t+1}$
is the
maximal subinterval of $\bZ$ containing $R_t=L_t$ on which
$\lambda_{t+1}\equiv0$. Observe that for each $t<T$ we have $\lambda
_t\equiv
0$ on $I_t$, while $\lambda_t(L_t-1)=\lambda_t(R_t+1)=1$. This
follows from
the CA rule for $\lambda_t$ by induction on $t$; the key observation
is that
if $|I_t|=1$ then $\lambda_{t+1}(L_t)=0$ (see Figure~\ref{wide-figure}).
Furthermore, any wide path started at $(0,0)$ is within $\bigcup_{t<T}
(I_t\times\{t\})$.

%
\begin{figure}

\includegraphics{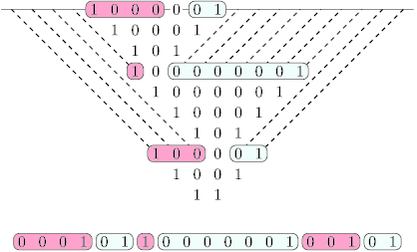}

\caption{The process of intervals of zeros used to prove
nonpercolation by
wide paths. Here, the interval lengths $(|I_0|,|I_1|,\ldots,|I_T|)$ are
$(5,3,1, 7,5,3,1, 4,2,0)$. The witness points are highlighted. The
corresponding binary sequence $(X_1,\ldots,X_{2T+1})$ is shown below;
it is
obtained by reading the states of the witness points in conventional text
order on the page, except with the left (red) intervals reversed. The dual
assignment of witness points to initial positions is shown via dashed lines
(witness points in the top row are assigned to themselves).}
\label{wide-figure}
\end{figure}

%
%
%
%

We now define an ordered sequence of $2T+1$ space--time points, which
we call
\df{witness} points, associated with the above sequence of intervals. If
$|I_t|=1$, we call $t+1$ a \df{refresh time}; we also declare 0 a refresh
time. Let $0=\tau_0<\tau_1<\cdots$ be the refresh times. We build the
sequence of witness points by appending certain points at each refresh time
$\tau_i$, in order. Specifically, for every $i$, we append all points in
$I_{\tau_i}\times\{\tau_i\}$, with the exception of $z_i=(L_{\tau_{i}-1},
\tau_i)$, in the following order: points to the left of $z_i$ in the
right-to-left order, followed by points to the right of $z_i$ in the
left-to-right order. Let $X_i=\lambda(s_i)$, where $s_1,\ldots, s_{2T+1}$
are the witness points in the order described. Write $X$ for the random finite
or infinite sequence given by $X=(X_1,\ldots, X_{2T+1})$ if $T<\infty
$ and
$X=(X_1,X_2,\ldots)$ if $T=\infty$. Our goal is to show that $X$ is
equal in
distribution to a sequence of independent fair coin flips stopped at a
certain a.s. finite stopping time.

Let $Y_1,Y_2,\ldots$ be independent random variables taking values $0$
and $1$
with equal probability. Partition this sequence into blocks of the form
$0^a10^b1$, where $a,b\ge0$, and let $S\geq1$ be the location of the
endpoint of the
first such block of odd length. Then $S$ is a.s. finite and
%
%
\begin{equation}
\label{wide-eq2} \P(S\ge t)< e^{-ct},
\end{equation}
since $S$ is at most the waiting time for the pattern $11011$. Write
$Y':=(Y_1,\ldots, Y_S)$. We claim
%
%
\begin{equation}
\label{wide-eq3} X\eqd Y'.
\end{equation}
Once \eqref{wide-eq3} is proved, the exponential bound \eqref{wide-eq2}
implies \eqref{wide-eq1}, since $ 2T+1\eqd S$.

We now proceed to prove \eqref{wide-eq3}. Since $\P(S<\infty)=1$,
the random
variable $Y'$ has countable support, thus it suffices to show that
$\P(X=y)=\P(Y'=y)$ for any $y$ with $\P(Y'=y)>0$. Choose such a $y$. The
event $\{X=y\}$ determines $T$ and $I_0,\ldots, I_T$ and, therefore, the
locations of the witness points. It follows that the event $\{X=y\}$ is
precisely the event that $\lambda$ takes the specified values $y$ on these
(deterministic) witness points. Now we use a dual assignment to show via
Proposition~\ref{dual-assignment}
\[
\P(X=y)= \bigl(\tfrac{1}2 \bigr)^{\mathrm{length\ of\ }y},
\]
as required. Considering the witness points in their order, we assign
to each
$(x,t)$ either $x+t$ or $x-t$, depending on whether it is to the right or
left of $z_i$ in its interval. (See Figure~\ref{wide-figure}.) This is
a dual
assignment simply because $\lambda^\bullet_t(\pm t)=1$ and
$\lambda^\bullet_t(x)=0$ for $x\in[-t,t]^C$.
\end{pf*}

\subsection{$\theta$-free paths}

We now discuss various aspects of $3$-free and $4$-free paths. The results
of this subsection are not needed for the proofs of
Theorems \ref{showcase}--\ref{stability}.

Since $3$-free paths are wide (Lemma~\ref{3-free}), the exponential bound
\eqref{exponential-bound} holds for $3$-free paths and there are no infinite
$3$-free paths when $\lambda_0$ is uniformly random on $\bZ$. Do such paths
exist started from special initial conditions? Indeed they do, as shown by
our next result.

Define the sequence of \df{principal voids} $V_0,V_1,\ldots$ of
$\lambda^\bullet$ according to Figure~\ref{principal-voids} (so that $V_{2j}$
and $V_{2j+1}$ have width $2^{j+1}-1$ and respective start times
$2\cdot2^j$
and $3\cdot2^j$). For $L>0$, let $W_i^L=V_i\cap(V_i+(L,0))$ [the notation
means that $V_i$ is translated by the vector $(L,0)$], and observe that $W_i$
is filled with $0$s when \add\ is started from any seed on $[0,L]$.
%
%
\begin{figure}[b]

\includegraphics{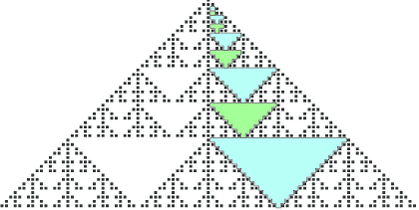}

\caption{The sequence of principal voids $V_0, V_1,\ldots, V_8$, numbered
from top to bottom; $V_0$ and $V_1$ each consist of a single point.}
\label{principal-voids}
\end{figure}

See Figure~\ref{exceptional} for an illustration of the next result,
and of
Corollary~\ref{rule-30-chaos} below. 
%
%
%
\begin{figure}

\includegraphics{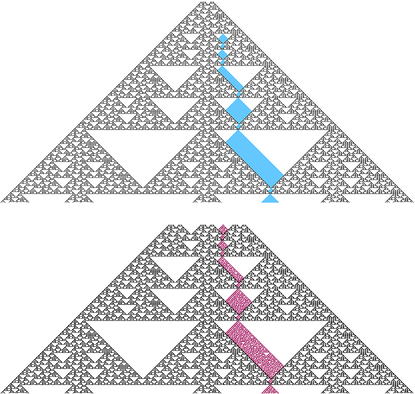}

\caption{Top: infinite 3-free paths (starting at time $64$) in \add\
from a seed supported in $4\bZ$. Bottom:
corresponding evolution of \textit{Web-adapted Rule 30}, started $64$
time units later with a single $2$ added,
and containing successively larger portions of \textit{Rule 30} evolution.}
\label{exceptional}
\end{figure}

%
\begin{prop}[(Exceptional percolation)]\label{exponentially-many}
Assume $\lambda_0$ is a seed on $[0,L]$ that vanishes outside $4\bZ$, where
$L$ is a multiple of $4$. Define the sets $W_i=W_i^L$ as above. Let $i$ be
such that $L\leq2^{\lfloor i/2\rfloor}-2$. From every point in $W_i$ there
is a $3$-free path to some point in $W_{i+1}$. In particular, from any such
point there is an infinite $3$-free path.
\end{prop}

\begin{pf}
For a maximal interval $I_t\subset\bZ$ on which $\lambda_t$
vanishes, define
its \df{successor} $I_{t+1}$ to be the maximal interval on which
$\lambda_{t+1}$ vanishes and such that $I_t\cap I_{t+1}\ne\varnothing
$; if
such an interval does not exist, let $I_{t+1}=\varnothing$.

Observe that $\lambda_t$ vanishes outside $2\bZ$ at all even times
$t$. Using
this, it is easy to verify by case-checking that every maximal interval of
$0$s has odd length at every time. It follows that any nonempty maximal
interval of $0$s has a nonempty successor. Similarly, since $\lambda_t$
vanishes outside $4\bZ$ when $t$ a multiple of $4$, it is easily verified
that for any $t$, from every point in $I_t\times\{t\}$ there is a $3$-free
path to some point in $I_{t+1}\times\{t+1\}$.

Let the \df{apex} $w_i$ be the bottommost point of $W_i$ (which is unique
since $L$ is even). Clearly, from every point in $W_i$ there is a $3$-free
path to $w_i$. Moreover, by the above, there is a $3$-free path from
$w_i$ to
$\bZ\times\{t_{W_{i+1}}\}$, where $t_{W_{i+1}}$ is the time of the top
interval of $W_{i+1}$. Finally, the condition $L\leq2^{\lfloor
i/2\rfloor}-2$ ensures that this top interval contains the
intersection of
the forward cone of $w_i$ with $\bZ\times\{t_{W_{i+1}}\}$.
\end{pf}

In fact, as a consequence of our next result, there are exponentially many
(in~$L$) seeds on $[0,L]$ whose forward cone contains an infinite $3$-free
path starting at time~$0$. As a result, as discussed in Section~\ref{prelim}
we can construct an exponential family of binary seeds which, after a
suitable replacement of some 0s by 2s, yield seeds for the \textit{Web-adapted
Rule 30} CA that are ``as chaotic'' as \textit{Rule 30}. We now make this
precise. Let $\cE_T$ denote the configuration of \textit{Rule 30}
started from
a single~$1$ at the origin, restricted to the region $\{(x, t)\dvtx
|x|\le
t\le
T\}$, and with all $1$s changed to~$2$s.

%
\begin{cor}[(Chaos in \textit{Web-adapted Rule 30})]\label{rule-30-chaos}
There exist at least $c_1 \exp(c_2 L)$ binary seeds on $[0,L]$ with the
following property. Each of these seeds has a location $z\in[0,L]$ occupied
by a 0; if this 0 is changed into a 2, the resulting seed for \textit
{Web-adapted Rule 30} generates a configuration that contains, for all $T$,
a~translated copy of $\cE_T$ within the first $c_3 T$ time steps. Here,
$c_1,c_2,c_3$ are absolute positive constants.
\end{cor}

\begin{pf}
Let $k$ be an integer. Denote the interval directly above (and of the same
length as) the top interval of $W_i^{4k}$ by $J_i^{4k}$. Suppose first that
$\lambda_0$ is an arbitrary configuration that vanishes off $4\bZ$ on
$[0,4k-4]$, and then on $[4k-3,4k]$ is chosen to be either $0000$ or $0001$
so as to ensure that for all $i$ the configuration in $J_i^{4k}$ is a subword
of $(110)^\infty$. That this can be achieved follows from
Lemma~\ref{predecessor-of-zeros} and additivity.

Let $i$ be such that $4k\leq2^{\lfloor i/2\rfloor}-2$, and consider the
sequence of successor intervals of the top interval of $W_i$. Using
Proposition~\ref{exponentially-many} (and its proof), all such
successors are
nonempty, and the successor at the time of $J_j$ (for $j>i$) consists
of a
single point. Furthermore, this point is at most $4k$ away from the center
of $J_j$ (for odd $j$) or from the center of the left half of $J_j$
(for even
$j$). These points will form the starting points of the translated
copies of
$\cE_T$.

Our set of binary configurations is the set of all resulting $\lambda_{t}$,
where $t$ is the time of the top interval of $W_i$; we define $z$ to be the
center (say) of this interval. This gives an exponential family of
configurations because the map on $\{0,1\}^\bZ$ corresponding to one
step of
\add\ is injective when restricted to seeds. This is easily verified (see
Section~\ref{ethers} for more information).

To conclude, we need the following properties of \textit{Web-adapted Rule
30},
which are easily checked from its definition.
\begin{longlist}[(ii)]
%
\item[(i)] If $1$s are initially confined to sites in $4\bZ$, and if at
some time
a maximal interval of $0$s and $2$s contains at least one $2$,
so does its successor interval.
\item[(ii)] Initial states $(110)^\infty112(110)^\infty$ and $0^\infty
20^\infty$, with the 2 at the origin, result in the same
configuration on $\bZ\times[1,\infty)$.\quad\qed
\end{longlist}
\noqed\end{pf}

We remark that a variant of Corollary~\ref{rule-30-chaos} may be
proved in
which we allow any finite set of 0s, in addition to the one at location $z$,
to be changed to 2s. The only change in the conclusion is that $c_3$ now
depends on the seed. This follows from two additional observations. First,
in \add, from any given set of maximal intervals of $0$s in $\lambda
_0$, the
number of successors cannot increase over time, and must thus eventually
stabilize. Second, if two \textit{Web-adapted Rule 30} seeds agree on
$[a,b]$, and have $1$s in $a$ and $b$ but no $1$s outside $[a,b]$, then their
configurations agree on the forward cone of $[a,b]$.

We conclude this section with the following conjecture supported by
computer experiments.

%
\begin{conjecture}\label{4-free}
The exponential bound \eqref{exponential-bound} in
Theorem~\ref{no-percolation} holds when diagonal path is replaced by 4-free
path.
\end{conjecture}

\section{Supercritical percolation}\label{supercritical}

In this section, we consider empty paths from random initial
conditions, and
in particular we prove the percolation result Theorem~\ref
{percolation}. The
results of this section are not needed for the proofs of
Theorems~\ref{showcase}--\ref{stability}, but they are of independent
interest and complement those of the previous section. We also consider
initial conditions where the randomness is restricted to the half-line
or a
finite seed. Here, many questions are open, but we establish some preliminary
results. The questions we consider are relevant to further understanding
certain web CA behavior.

\subsection{Percolation of empty paths}

As Figure~\ref{fluidempty} suggests, the set of points reachable by empty
paths emanating from an interval at time $0$ form an interval at each
subsequent time. With random initial conditions, this interval spreads
linearly provided it survives. Proving this is the key to
Theorem~\ref{percolation}.
%
%
\begin{figure}

\includegraphics{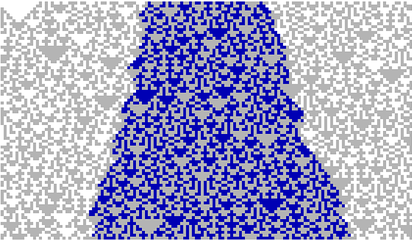}

\caption{All empty paths from an interval at time $0$ are highlighted
in blue.}
\label{fluidempty}
\end{figure}

Suppose $\lambda_0$ is given. The \df{rightward $Z$-path} from a space--time
point $(x,t)$ is an infinite sequence of points $(r_s,s)$, $s\geq t$ defined
as follows. Start with $r_t=x$. Inductively, let $r_{s+1}$ be the largest
integer $y$ in $(-\infty, r_s+1]$ for which $\lambda_{s+1}(y)$ is $0$;
or if
there is no such $y$ we take $r_u=-\infty$ for all $u>s$. Note that
$\lambda(r_s,s)=0$ for all $s>t$ for which $r_s$ is finite, but not
necessarily for $s=t$. Analogously, we define the \df{leftward $Z$-path}
$(\ell_s,s)$, $s\geq t$ by reversing the space coordinate in the definition.

%
\begin{lemma}[(Properties of $Z$-paths)]\label{empty-Z-connection}
Suppose $\lambda$ is \add\ from any initial configuration.
\begin{longlist}[(iii)]
%
\item[(i)] Suppose $\lambda(0,0)=0$ and let $(r_t,t)$, $t\ge0$ be the rightward
$Z$-path from $(0,0)$. If $x\leq r_t$ and $\lambda(x,t)=0$ then there is
a empty path from $(-\infty,0]\times\{0\}$ to $(x,t)$.
\item[(ii)]
Fix an interval $[a,b]$ with $a\le b$. Let $(\ell_t,t)$ be the leftward
$Z$-path from
$(a,0)$, and $(r_t,t)$ the rightward $Z$-path from $(b,0)$. Suppose
that $\ell_s\le r_s$ for
every $s\leq t$. Then for any $y\in[\ell_t, r_t]$ with $\lambda_t(y)=0$,
there is an empty path from $[a-2,b+2]\times\{0\}$ to $(y,t)$.
\item[(iii)] Under the assumptions of {\rm(ii)}, suppose also that
$\lambda_0(a)=\lambda_0(b)=0$. Then for any $y\in[\ell_t, r_t]$ with
$\lambda_t(y)=0$,
there is an empty path from $[a,b]\times\{0\}$ to $(y,t)$.
\item[(iv)] Conversely, if
there is an empty path from $[a,b]\times\{0\}$ to some $(y,t)$, then
$\ell_s\le r_s$ for all $s\leq t$, and $\ell_t\le y\le r_t$.
\end{longlist}
\end{lemma}

\begin{pf}
We omit the proof of (i), as it is
similar to the proof of (iii), which proceeds by induction as follows.
The argument reduces
to verifying (iii) at time $t=1$.
Assume $\ell_1\le r_1$ and take $y\in[\ell_1,r_1]\subseteq[a-1,b+1]$
with $\lambda_1(y)=0$. Then there exists an $x\in\{y-1,y,y+1\}$ with
$\lambda_0(x)=0$.
It remains to verify that $x$ can be chosen to be in $[a,b]$.
If $y\in[a+1,b-1]$ this is clear; if $y\in\{b, b+1\}$, we may take $x=b$
and if $y\in\{a, a+1\}$ we may take $x=a$.

The above argument also proves (ii): we verify the claim at time $t=1$ and
then use (iii). The last claim (iv) is an easy consequence of
definitions of
empty and $Z$-paths.
\end{pf}

The key fact in establishing percolation of empty paths is that $r_t$ has
drift $1/4$. The proof is somewhat similar to that of nonpercolation for
wide paths, Theorem~\ref{no-percolation}.

%
%
\begin{lemma}[(Drift)] \label{r-drift}
Suppose that the initial configuration $\lambda_0$ is uniformly random on
$\bZ\setminus\{0\}$ and $\lambda_0(0)=0$. Let $(r_t,t)$, $t\ge0$ be the
rightward $Z$-path from $(0,0)$. For every $\varepsilon>0$, there
exists a
constant $c=c(\varepsilon)>0$ so that $\P(|r_t-t/4|>\varepsilon t)<e^{-ct}$.
\end{lemma}

\begin{pf}
We first describe an exploration process that determines the rightward
$Z$-path from the origin $(0,0)$. We designate $(0,0)$ to be the first
\df{refresh point}. Now we examine the states of the points
$(1,1),(0,1),(-1,1),\break (-2,1),\ldots,$ in this order, until we find the first
point with state $0$. Let $G$ be the number of points examined, and call
them \df{witness points}. Since the states of these witness points are
$01\cdots1$ (from left to right), certain states at the immediately
following time steps are determined. Specifically, the pattern
$01\cdots1$
is immediately followed by patterns of the same form, but with the length
decreasing by $2$ at each step and centered at the same location,
ending with
either $01$ or $0$ according to whether $G$ was even or odd. (See
Figure~\ref{empty-figure}.) We designate the location of the $0$ in this
last pattern to be the next refresh point. It is $(1-\lfloor G/2\rfloor,
\lceil G/2\rceil)$. Now iterate the process starting at the new refresh
point. Note that the rightward $Z$-path from $(0,0)$ consists precisely of
the $0$s at the left ends of the $01\cdots1$ patterns, including the refresh
points. Observe also that the $Z$-path is determined by the locations
of the
refresh points, and that these are determined by examination of the witness
points.

%
\begin{figure}

\includegraphics{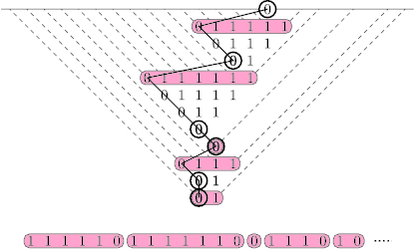}

\caption{The rightward $Z$-path (solid lines) from the origin,
together with
its refresh points (circled), and witness points (highlighted in red). A
dual assignment of the witness points to initial positions is indicated by
the dashed lines.
The states of the witness points in the order they are examined are
shown below.}
\label{empty-figure}
\end{figure}

Now consider the above exploration process for the initial
configuration that
is uniformly random on $\bZ\setminus\{0\}$ and $0$ at $0$. Let
$(X_i)_{i\geq1}$ be the sequence of states of the witness points, in the
order that they are examined by the exploration process. We claim that
$(X_i)_{i\geq1}$ is uniformly random. It suffices to check that
$(X_1,\ldots,X_n)$ is uniformly random. This follows from
Proposition~\ref{dual-assignment}, by the dual assignment in which a witness
point $(x,t)$ is assigned to $x+t$ if it is the rightmost witness point in
its $01\cdots1$ pattern, and otherwise to $x-t$. See
Figure~\ref{empty-figure}.

%
%
%
%
%

Let $G_i$ be the number of witness points examined in the row
immediately below the $i$th refresh point. Then $(G_i)$ are i.i.d.
Geometric($1/2$) random variables. Furthermore, the sequence of refresh
points is a random walk on $\bZ^2$ with steps $(1-\lfloor G_i/2\rfloor,
\lceil G_i/2\rceil)$. As $\mathbf{E}\lfloor G_i/2\rfloor=2/3$ and
$\mathbf{E}\lceil G_i/2\rceil=4/3$, each step has expectation vector $(1/3,
4/3)$. The proof is concluded by standard large deviation estimates.
\end{pf}

\begin{pf*}{Proof of Theorem~\ref{percolation}}
For $L$ to be chosen later, consider the leftward path $(\ell_t,t)$ started
at $(-L,0)$ and the rightward path $(r_t,t)$ started at $(0,L)$. Then,
by a
union bound and symmetry,
\[
\P (\ell_t<r_t\ \forall t ) \ge\P (
\ell_t\le-1\mbox{ and }r_t\ge1\ \forall t ) \ge1-2 \P
(r_t\le0\mbox{ for some $t$} ). %
\]
By Lemma~\ref{r-drift}, for $L$ large enough we have $\P(r_t\le
0\mbox{ for
some $t$})$ $\le1/3$. Call a site $x\in\bZ$ \df{good} if an
infinite empty
path starts at $(x,0)$. Thus, by Lemma~\ref{empty-Z-connection}(ii),
\[
\P \bigl( [-L-2, L+2] \mbox{ contains some good site} \bigr) \ge\P (
\ell_t<r_t\ \forall t )\ge1/3.
\]
Consequently, by translation-invariance, $\P(0\mbox{ is good})\ge
1/[3(2L+5)]$.
\end{pf*}

\subsection{Empty paths for half-lines and seeds}

How do empty paths behave when the initial configuration is a random seed?
This question is largely unresolved. (In contrast, the next section will
provide detailed answers for diagonal and wide paths.) A first step
would be
to understand the case of a uniformly random half-line, for which the
following conjecture is natural given Lemma~\ref{r-drift}.

%
\begin{conjecture}
Suppose the initial condition $\lambda_0$ is uniformly random on
$[1,\infty)$
and 0 elsewhere. Let $(r_t,t)$, $t\geq0$ be the rightward $Z$-path from
$(0,0)$. Then $r_t/t\to1/4$ as $t\to\infty$.
\end{conjecture}

We prove that a much weaker statement holds deterministically: for an initial
configuration supported in a half-line, empty paths penetrate
arbitrarily far
into its forward cone.

%
\begin{lemma}[(Unbounded penetration)]\label{zeros-get-in}
Assume that the initial condition $\lambda_0$ has no $1$s outside
$[1,\infty)$. Let $(r_t,t)$ be the rightward $Z$-path from $(0,0)$. Then
$\sup_t (r_t+t)=\infty$.
\end{lemma}

\begin{pf}
We first observe that for any initial configuration $\lambda_0$ of
\add, if
$(x,t)$ has state $0$ and $t\geq1$ then at least one of the three points
$(x,t-1),(x\pm1,t-1)$ has state $0$ also. Iterating this, we see that there
must be an empty path from $\bZ\times\{0\}$ to $(x,t)$. We call any such
path an \df{ancestral path} of $(x,t)$.

Now, under the conditions of the lemma, note that for any $m\geq0$, the
sequence of configurations on the intervals $[-t,-t+m+1]\times\{t\}$ is
periodic in $t$ starting from some time $t_p$ depending on $m$ and
$\lambda_0$. For $a\ge0$, define the leftward diagonal $D_a:=\{
(a-t,t)\dvtx t\geq
0\}$. Then $\lambda$ cannot be identically $1$ on two consecutive diagonals
$D_a$ and $D_{a+1}$, and also cannot be identically $1$ on $D_a$ and
identically $0$ on $D_{a+1}$. (Indeed, in either case we deduce that
$\lambda$ is also $1$ on $D_{a-1}$, leading to a contradiction by induction.)

Fix $m\ge0$. We will show that for some $t$ there an empty path from
$(-\infty,0]\times\{0\}$ to $\{(-t+m+1,t), (-t+m,t)\}$, which
suffices by
Lemma~\ref{empty-Z-connection}. To verify this claim, we may assume
that the
periodic orbit commences initially, that is, that $t_p=0$. There must
be a time
$t$ with either $\lambda_t(-t+m+1)=0$ or $\lambda_t(-t+m)=0$; by periodicity
there must be infinitely many such times. Now take the \textit{leftmost}
ancestral path
of one of these two points.
Suppose this path does not start on $(-\infty,0]\times\{0\}$. Then,
if $t$
is large enough, the path has a diagonal segment longer than the period of
the orbit; additionally, all states immediately to the left of such a segment
must be 1. By periodicity, we have, for some $a\in[0,m+1]$, infinite
diagonals $D_a$ and $D_{a+1}$ on which $\lambda$ is identically 1 and 0,
respectively. This is in contradiction with our observations above.
\end{pf}

We remark that the supremum in the above lemma cannot be replaced with a
limit; a counterexample is $\lambda_0\equiv1$ on $[2,\infty)$ and
$\lambda_{0}(1)=0$.

Returning to our earlier question on seeds, Figure~\ref{fluid} (top) shows
the set of all points on empty paths from $(-1,0)$, when $\lambda_0$
is a
random seed on $[0,25]$. We believe that for typical long seeds, the right
frontier of this set lags behind the right edge of the forward cone of the
seed by a nontrivial power of $t$ in the limit $t\to\infty$. This is a
natural guess, since the frontier has speed $1$ in the voids, but presumably
speed $1/4$ on the fractal set occupied by randomness. It appears plausible
that such a process is a driving force behind the evolution of some
exceptional seeds for web CA, including the examples in Figures \ref{webxor}
(bottom), \ref{mod-webxor} (bottom) and possibly \ref{web1or3}.
%
%
\begin{figure}

\includegraphics{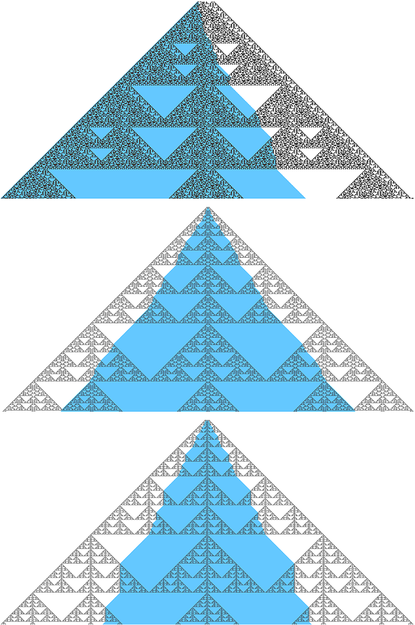}

\caption{The set of all points (blue) on empty paths starting from
certain initial points,
in \add\ started from three different seeds: two apparent power-law
cases, and a devil's staircase.}
\label{fluid}
\end{figure}

We believe that similar power law behavior holds for some specific small
seeds. One example is shown in Figure~\ref{fluid} (middle): the seed is
$1 000 \widehat{0} 000 1$, and empty paths from the middle
$\widehat{0}$
are highlighted. However, some seeds exhibit entirely different behavior.
The bottom picture shows the empty paths from the two $\widehat{0}$s
in the
seed $1 000 \widehat{0} 000 1 000 \widehat{0} 000 1$. Despite
apparent initial similarity to the previous case, here the rightmost point
$(r_t,t)$ reachable at time $t$ has $r_t/t$ bounded strictly between
$0$ and
$1$ at $t\to\infty$. Indeed, the rescaled path $2^{-n}\{(r_t,t)\dvtx
t\geq
0\}$
converges as $n\to\infty$ to a variant of the Cantor function or ``devil's
staircase.'' This may be proved by an inductive scheme.

As a preliminary step toward the power law behavior postulated above, we
prove a version in a simplified setting. Recall from
Section~\ref{prelim} that $\mu$ denotes the \textit{Xor} additive cellular
automaton rule.
Given a configuration $\mu\in\{0,1\}^{\bZ\times[0,\infty)}$, we define
the \df{$\chi$-path} starting
from a point $(x,0)$ to be the sequence of points $((x_t,t)\dvtx t\geq
0)$ given
by $x_0=x$ and
\[
x_{t+1}=\cases{ x_t, & \quad$\mu(x_t,t)=1$;
\vspace*{2pt}
\cr
x_t+1, & \quad $\mu(x_t,t)=0$. }
\]
In other words, the path makes a down step from a $1$, and a diagonal step
from a $0$. This is intended as a simplified model for a rightward $Z$-path,
which moves with speed $1$ in $0$s, but with a slower speed in a random
configuration.

%
%
\begin{prop}[(Power law for \textit{Xor})] \label{power-law}
Let $\mu$ be the \textit{Xor} CA with initial configuration
$\mu_0$ equal to $1$ on the two-point set
$\{-1,0\}$ and $0$ elsewhere. The $\chi$-path $((x_t,t))_{t\geq0}$ starting
from $(0,0)$ satisfies
\[
x_t = t-\Theta \bigl(t^{\log2/ \log3} \bigr) \qquad\mbox{as }t\to \infty.
\]
\end{prop}
%

\begin{pf}
We first note some easy facts about $\mu$. Denote the interval of points
$R(k,t):=((i,t)\dvtx t-2^k<i\leq t)$ on the right side of the forward cone
of the origin. For any $k\geq1$, the state-vectors $(\mu(z)\dvtx z\in R(k,t))$
on these intervals form a periodic sequence in $t$ with period $2^{k-1}$.
Furthermore, the sequence of state-vectors on the intervals
$R(k+1,t)\setminus R(k,t)$ consists precisely of the all-$0$ vector repeated
$2^{k-1}$ times followed by the first $2^{k-1}$ state-vectors for $R(k,t)$
(all repeated with period $2^k$). See Figure~\ref{log-path} for an
illustration of the case $k=2$.
%
%
\begin{figure}[b]

\includegraphics{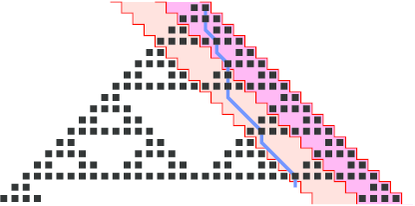}

\caption{The $\chi$-path from the origin in the Xor cellular automaton,
together with the construction used in its analysis. The strips
$S(2)$ (pink) and $S(3)\setminus S(2)$ (orange) are shaded.}
\label{log-path}
\end{figure}

Let $E_k:=\min\{t\geq0\dvtx(x_t,t)\notin R(k,t)\}$; this is the time
at which
the $\chi$-path leaves the diagonal strip $S(k):=\bigcup_t R(k,t)$.
This can
only happen at a down step, which can occur only at a $1$ of $\mu$ in the
leftmost diagonal of $S(k)$. It follows that $E_k$ is divisible by
$2^{k-1}$; write $E_k=2^{k-1} e_k$. For example (referring to
Figure~\ref{log-path}), we have $e_2=3$ and $e_3=4$.

In order to leave the strip $S(k+1)$, the path must first leave $S(k)$, and
then leave $S(k+1)\setminus S(k)$. By the above observations on periodicity,
and the fact that the path moves diagonally on $0$s, we deduce that
\[
e_{k+1} = \biggl\lfloor\frac{e_k}{2} \biggr\rfloor+e_k.
\]
The proof is complete using induction and obvious monotonicity
properties of the $\chi$-path.
\end{pf}

Among many unresolved questions, we do not know
whether an analogue of Proposition~\ref{power-law} holds when the
$\chi$-path
is defined similarly in terms of the \add\ CA $\lambda$ rather than
$\mu$.

\section{Additive dynamics from random seeds}\label{seeds}

Our goal in this section is to transfer the nonpercolation results for
infinite random initial configurations to random seeds. The proofs exploit
an intriguing interplay between randomness and periodicity in the
configuration started from a random seed.

%
\begin{lemma}[(Random edge-intervals)]\label{diagonal-randomness}
Assume $\lambda_0$ is a uniformly random binary seed on $[0,L]$.
For a fixed $t$, the state on the interval $[t,t+L]\times\{t\}$ is
uniformly random.
\end{lemma}

\begin{pf}
This is an immediate consequence of Lemma~\ref{boundary-periodicity}(i)
and Corollary~\ref{one-zeros}.
\end{pf}

%
\begin{lemma}[(Edge-periodicity)]
\label{diagonal-periodicity} For any $\lambda_0$ which is 0 on
$[L+1,\infty)$, and
any $k\ge1$, the sequence of edge configurations
$(\lambda(i,t)\dvtx i=t+L-k+1,\ldots, t+L)$ is periodic in
$t$, with period at most $2k$.
\end{lemma}

\begin{pf}
This follows from Lemmas~\ref{boundary-periodicity}(iv) and~\ref
{duality}.
\end{pf}


Our first result establishes that, in subcritical cases, paths from the
initial state do not reach far into the forward cone of $[0,L]\times\{
0\}$.
This is illustrated in Figure~\ref{blocked-fluid-seed}, in which
$L=25$ and
all points on paths from $\bZ\times\{0\}$ are again depicted in blue (only
one layer of points outside the forward cone is colored blue, as all
such are
trivially reachable from $\bZ\times\{0\}$).
%
%
\begin{figure}

\includegraphics{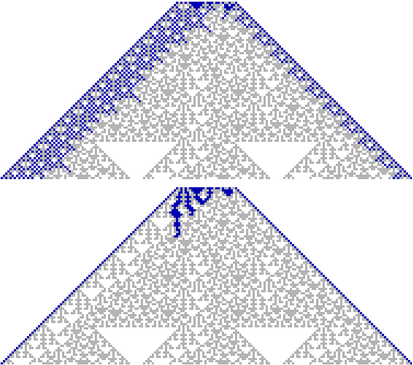}

\caption{Illustrations of the set of all points reached by paths starting
within an initial random seed.
Top: empty diagonal paths; bottom: wide paths.}
\label{blocked-fluid-seed}
\end{figure}

%
\begin{prop}[(Percolation into the cone)]
\label{no-outside-influence}
Suppose $\lambda_0$ is a uniformly random binary seed on $[0,L]$.
The probability that there is an empty diagonal path from $\bZ\times\{
0\}$
to the forward cone of $[0,L]\times\{\lfloor C\log L\rfloor\}$ goes to
0 as
$L\to\infty$. The same is true for wide paths. Here, $C$ is an absolute
constant.
\end{prop}

\begin{pf}Let $k$ be a positive integer to be chosen later satisfying $2k+1<L$.
Call a space--time point $(x,t)$ \df{bad} if there exists an empty diagonal
path from $\bZ\times\{t-k\}$ to $(x,t)$. If the state on the interval
$I(x,t):=[x-k,x+k]\times\{t-k\}$ is uniformly random, then
Theorem~\ref{no-percolation} implies that $\P((x,t)\mbox{ is
bad})\le
\exp(-ck)$ for an absolute constant $c>0$.

%
\begin{figure}[b]

\includegraphics{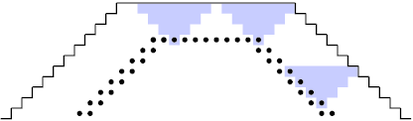}

\caption{An illustration of the proof of
Proposition \protect\ref{no-outside-influence}. The outline of the
forward cone of
$[0,L]\times\{0\}$ is shown by the solid line (here $L=16$). Points in the
set $S$ are shown as black discs. Any path from the top row to the region
below $S$ must pass through $S$. For each point in $S$, the top row of the
associated triangle (of size $k=3$) has a uniformly random state (three such
triangles are shaded).} \label{no-outside-influence-figure}
\end{figure}

%
%
%
%
%
%
%
We define an infinite set of points $S$ via
Figure~\ref{no-outside-influence-figure}. This set has the following
properties: (i) any path from $\bZ\times\{0\}$ to the forward cone of
$[0,L]\times\{2k\}$ must pass through a point in $S$; and (ii) for every
$(x,t)\in S$, the state on the interval $I(x,t)$ defined above is uniformly
random, either trivially or by Lemma~\ref{diagonal-randomness}.

We wish to bound the probability that $S$ contains a bad point by a union
bound. The set $S$ is infinite, but Lemma~\ref{diagonal-periodicity} implies
that the states of the relevant intervals $I(x,t)$ for $(x,t)$ in the
diagonal ``arms'' of $S$ repeat with period at most $2(2k+2)$. Thus, besides
the at most $L$ points on the top section of $S$, there are only $8(2k+2)$
distinct cases to consider. Hence,
\[
\P(\mbox{$S$ contains a bad point})\le(L+16k+16) e^{-ck}.
\]
The proof is completed by taking $k=\lfloor C'\log L\rfloor$ for a suitably
large $C'$ (the argument for wide paths is identical).
\end{pf}

Similarly, we next show that to each void of $\lambda^\bullet$ there
corresponds a periodic strip that blocks diagonal and wide paths. Fix a void
$V$ of $\lambda^\bullet$, and an integer $L\ge0$. Define the \df{perturbed
void} $V$ to be the triangular region
\[
W_L(V)=V\cap\bigl(V+(L,0)\bigr).
\]
See Figure~\ref{perturbed-void-figure} for an example. Note that
$W_L(V)=\varnothing$ unless the width of $V$ exceeds~$L$. Further, fix an
integer $m\ge1$, assume that the top interval of $W_L(V)$ is
$[a,b]\times
\{t\}$, and define the following interval above $W_L(V)$:
\[
J_{L,m}(V)=\bigl[a-2^m, b+2^m\bigr]
\times{t-2^m}.
\]
[We set $J_{L,m}(V)=\varnothing$ when $W_L(V)=\varnothing$.]
%
%
\begin{figure}[b]

\includegraphics{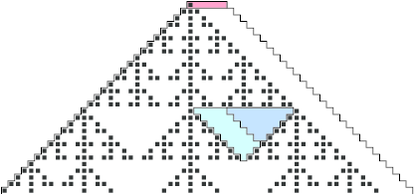}

\caption{Perturbed void (dark blue), with $L=5$, of the void with top
interval $[1,15]\times\{16\}$.
The perturbed void is filled with $0$s for any seed included in the
interval of six red
points. The forward cone of this interval is outlined.}
\label{perturbed-void-figure}
\end{figure}

%
\begin{lemma}[(Periodic and random intervals above voids)]\label
{above-voids-general}
Suppose the initial configuration $\lambda_0$ vanishes outside $[0,L]$.
Let $m$ be a nonnegative integer. Let
$V$ be a void of $\lambda^\bullet$ of width at least $L$ and at least $2^m$.
%
%
\begin{longlist}[(iii)]
%
\item[(i)]$\lambda$ vanishes on $W_L(V)$.

%
\begin{figure}

\includegraphics{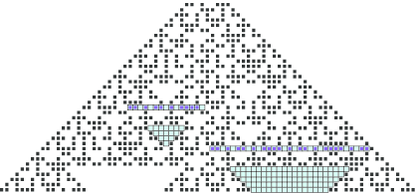}

\caption{Illustration of Lemma \protect\ref{above-voids-general}\textup{(ii)}
with the seed $110100111$
on $[0,8]$ and $m=2$ and the same two voids as in Figure \protect\ref
{above-voids-figure};
the repeating string is $A=010110101111$.}
\label{above-voids-general-figure}
\end{figure}
\item[(ii)] There exists a string $A$ of length $3\cdot2^m$, depending on
$m$ and $\lambda_0$ but not on~$V$, such that the configuration of
$\lambda$ on $J_{L,m}(V)$ is a subword of $A^\infty$.
\item[(iii)] Now suppose that $2^{m+1}\le L$ and that $\lambda_0$ is uniformly
random on $[0,L]$. Then every interval of length $2^m$ in $J_{L,m}(V)$ has
uniformly random state.
\end{longlist}
%
%
\end{lemma}

\begin{pf} Claim (i) is a simple consequence of Lemma~\ref{duality},
(ii) follows from Proposition~\ref{above-voids} and Lemma~\ref
{duality}, and
(iii) from Proposition~\ref{above-voids} and Corollary~\ref{one-zeros}.
(See Figure \ref{above-voids-general-figure}.)
\end{pf}

%
\begin{prop}[(Percolation into voids)]\label{no-crossing-into-void}
Assume the conditions
in Lem\-ma~\ref{above-voids-general}\textup{(iii)}. Let $\cross(m)$ be the event that
there exists a void $V$ for which there is either an empty diagonal or
a wide
path from $J_{L,m}(V)$ to $W_L(V)$. Then
\[
\P\bigl(\cross(m)\bigr)\le\exp\bigl(-c2^m\bigr)
\]
for some universal constant $c$.
\end{prop}

\begin{pf} Using Theorem~\ref{no-percolation}, and Lemma~\ref
{above-voids-general}(ii) and (iii), this follows by a similar argument
to the proof of
Proposition~\ref{no-outside-influence}. The key point is that by
Lemma~\ref{above-voids-general}(ii), only $3\cdot2^m$ distinct cases
need to
be considered in the union bound.
\end{pf}

\section{Replication and ethers in web cellular automata}\label{ca}
We can now prove Theorems \ref{showcase} and
\ref{stability} from the \hyperref[sec1]{Introduction}.

\begin{pf*}{Proof of Theorem~\ref{stability}}
This is an immediate consequence of Proposition~\ref{no-outside-influence}
and Lemma~\ref{compliance}.
\end{pf*}

\begin{pf*}{Proof of Theorem~\ref{showcase}}
Consider a uniformly random binary seed on $[0,L]$. In the context of
Proposition~\ref{no-crossing-into-void}, let $M$ be the smallest $m$ with
$2^{m+1}\le L$, for which $\cross(m)$ does not occur. If such an $m$
does not
exist, let $M=\infty$. By Proposition~\ref{no-crossing-into-void},
$M$ is
tight as $L\to\infty$.

If $M=\infty$, take $R_L=\infty$. Assume now that $M<\infty$. Then there
exists a string $A'$ of $0$s and $2$s of length of $3\cdot2^m$ so that the
top row of $W_L$ is a segment of $(A')^\infty$ for every void. This holds
because the first level configuration $\lambda$ is periodic with the required
period on a strip above $W_L$, by Lemma~\ref{above-voids-general},
while the
absence of relevant paths makes the top row also periodic by
Lemma~\ref{compliance}. Moreover, by the same results, the periodic pattern
is the same for all voids.

Consider the CA $\xi$ started with a periodic configuration $B^\infty
$, for
some string $B$ of length $\sigma$. The evolution is periodic in time after
some initial burn-in time interval: let $T_B$ be the first time $t$
such that
$\xi_t$ equals $\xi_s$ for some $s>t$. Let $\burnin(\sigma)=\max_B
T_B$. Our
random distance $R_L$ is $\burnin(3\cdot2^M)+1$, and the proof is
finished by
Lemma~\ref{light-speed} (which requires the $+1$).
\end{pf*}

As remarked earlier, Theorem~\ref{showcase} implies that the union of the
regions that are filled by a translate of the ether has density $1$ within
the forward cone of the seed. Therefore, on the event that $R_L<\infty
$ the
set of nonzero points has a rational density within the same forward cone.
We do not know whether the same holds for \emph{every} seed.

For an arbitrary web CA satisfying the conditions of Theorem~\ref{showcase}, \break
$\limsup_{L\to\infty} \P(R_L\ge r)$ decays at least as fast as a power
law in
$r$. This is easily seen from the above proof, using the fact that
$\burnin(\sigma)\leq2^\sigma$. In cases when the dynamics
restricted to
$0$s and $2$s is additive, including \webxor, \textit{Modified
Web-Xor} and
\pig, one can easily show that the decay is exponential. Identical remarks
apply to the temporal period of the ether $\eta_L$.

\section{Bounds on ether probabilities}\label{ethers}

In this section, we prove Theorem~\ref{ether-births}, and explain how explicit
lower bounds on ether probabilities are proved. We also indicate how some
ethers can be ruled out for certain rules.

For $m\ge0$, we call the string $A$ in Lemma~\ref{above-voids-general}(ii)
the \df{level-$2^m$ link} of the seed $\lambda_0$. (Note that the
choice of
$A$ is unique up to periodic shifts.) Fix an integer $k\ge1$ and a binary
string $A$. Consider \add\ with initial periodic configuration
$\lambda_0=A^\infty$. If there is no empty diagonal (resp., wide,
$\theta$-free) path from $\bZ\times\{0\}$ to $\bZ\times\{k-1\}$ in the
resulting configuration $\lambda$, then we say that $A$ is a
\df{blocker} to \df{depth} $k$ for diagonal (resp., wide, $\theta$-free)
paths. If $\lambda_{k-1}\not\equiv0$, then we say that $A$ is
\df{nondegenerate} to depth $k$.

Fix an $m\ge1$, and let $\lambda_0$ be a uniformly random seed on
$[0,L]$, with $L\ge2^{m+1}$.
For diagonal and wide paths,
Proposition~\ref{no-crossing-into-void} implies that
\[
\P\bigl(\mbox{the level-$2^m$ link of $\lambda_0$ is a
blocker to depth $2^m$}\bigr)\ge1-\exp\bigl(-c2^m\bigr),
\]
for some universal constant $c$.

Further, consider a web CA $\xi_t$ with an ether
$\eta\in\{0,2\}^{\bZ^2}$. The \df{signature} of $\eta$ is a
string $B$ such that, for some $t$,
$\eta(\cdot, t)$ equals a (spatial) translation of $B^\infty$, and
is the first in the lexicographic order among shortest such strings. Observe
that two ethers are equivalent if and only if they have the same signature.
We say that a binary string $A$ \df{produces}
$\eta$ with signature $B$ if the initial state $\xi_0=A^\infty$
makes $\xi_t$ equal to a translation of $B^\infty$ at some time $t$.

%
\begin{lemma}[(Blockers)]\label{blocker-ether}
Let $\xi_t$ be a diagonal-compliant (resp.: wide-\break compliant,
or $\theta$-free-compliant) web CA. Further, let $A$ be a string
that is a blocker to depth $2^m$ for diagonal (resp.: wide, or
$\theta$-free) paths
and produces an ether $\eta$. If a seed $\xi_0$ results in the
level-$2^m$ link $A$,
then $\xi$ is a replicator with ether $\eta$. If,
in addition, the CA $\xi_t$ has no spontaneous birth, then $\eta
\equiv0$.
\end{lemma}

\begin{pf}
The first claim follows by the same arguments as in the proof
of Theorem~\ref{showcase}. The last claim follows by Lemma~\ref{compliance}.
\end{pf}

Denote by $\seeds_{[a,b]}$ the set of binary seeds that vanish outside
$[a,b]$, and by
$\seeds=\bigcup_{a\le b} \seeds_{[a,b]}$ the set of all binary seeds.
Let $g:\{0,1\}^{\bZ}\to\{0,1\}^{\bZ}$ be the map determined by one step
of the \textit{1 Or 3} rule (i.e., the map $\lambda_0\mapsto\lambda_1$).
It is well known (and easy to prove) that, for $a\le b$, the map
$g$ is injective from $\seeds_{[a,b]}$ to $\seeds_{[a-1, b+1]}$ and,
therefore,
the restriction $g|_{\seeds}$ is injective.
We say that a binary seed $\lambda_0$ \df{has a predecessor} if it is
in the image of $g|_{\seeds}$. More generally, $\lambda_0$ \df{has $k$
predecessors},
for $k\ge1$,
if it is in the image of the $k$th iteration $ (g|_\seeds
)^k$; in that case,
$ (g|_\seeds )^{-k}(\lambda_0)$ contains
a unique seed called the
\df{$k$th predecessor} of $\lambda_0$.
We denote by $\pred_k$ the set of all seeds that have $k$ predecessors.
The following lemma follows immediately from the properties of
$g|_{\seeds}$.

%
\begin{lemma}[(Predecessors of random seeds)]\label{predecessors-of-seeds}
Assume $\lambda_0$ is a uniformly random binary seed on $[0,L]$ and
that $1\le k\le L/2$. Then
$\P(\pred_k)=1/4^k$. Moreover, conditioned on $\pred_k$, the $k$th
predecessor of $\lambda_0$
is a uniform binary seed on $[k, L-k]$.
\end{lemma}

%
\begin{lemma}[(Predecessors and links)]\label{predecessors-and-links} For
$m\ge0$, a seed $\lambda_0$
has $2^m$ predecessors if and only if its
level-$2^m$ link is 0.
\end{lemma}

\begin{pf}
If $\lambda_0$ has $2^m$ predecessors, then its level-$2^m$ link is
$0$ by
Lemmas~\ref{boundary-periodicity}(ii) and~\ref{duality}.

Conversely, assume that $\lambda_0\in\seeds_{[0,L]}$ is given by a string
$S$ of length $L+1$. If $n$ is large enough so that $2^{n}>2L$ and
$2^{n}>2^{m+1}$, the configuration of $\lambda_{2^n}$ on $[0,L]$ is $S$,
again by Lemma~\ref{boundary-periodicity}(ii) and additivity. Recall that
the level-$2^m$ link is the same for all voids, on the left and on the right
of the vertical line $x=0$. If this link is $0$, the configuration on
$[0,L]$ at time $2^n-2^m$
provides a seed $\lambda_0'$, such that $g^{2^m}(\lambda_0')=\lambda_0$.
\end{pf}

%
\begin{lemma}[(Ether probabilities)]\label{positive-probability}
Suppose some seed $S_0\in\seeds_{[0,s-1]}$ is a
replicator with some ether $\eta$, and is such that for some $m\ge1$,
the level-$2^m$
link is a blocker to depth $2^m$. Let $\xi_0$ be a uniformly random binary
seed on $[0,L]$. Then
\[
\liminf_{L\to\infty}\P(\mbox{$\xi_0$ is a replicator
with ether $\eta$})\ge2^{-s- 2^{m+1}}.
\]
\end{lemma}

\begin{pf}
Assume a seed $S_1$ with support in $[s,\infty)$ has $2^m$ predecessors.
Form a seed $S$ by adding the configurations of $S_0$ and $S_1$.
Then, by Lemmas~\ref{predecessors-and-links} and~\ref{duality}, $S$
has the same level-$2^m$ link as $S_0$ and, therefore,
by Lemma~\ref{blocker-ether}, is a replicator with
the same ether $\eta$. In the rest of the proof, we apply this fact
to random seeds.

Suppose now $\xi_0$ is a uniformly random seed in $\seeds_{[0,L]}$,
with $L$ large enough so that
$L-s\ge2^{m+1}$. Let $\xi_0'$ (resp., $\xi_0''$) be the
random seed that agrees with $\xi_0$ on $[0,s-1]$ (resp., $[s,L]$) and
vanishes elsewhere. Then
\begin{eqnarray*}
\P(\mbox{$S$ is a replicator with ether $\eta$}) &\ge&\P\bigl(\mbox{$
\xi_0'=S_0$, and $\xi_0''$
has $2^m$ predecessors}\bigr)
\\
&=&\P\bigl(\xi_0'=S_0\bigr)\cdot\P\bigl(
\mbox{$\xi_0''$ has $2^m$
predecessors}\bigr)
\\
&=&2^{-s}\cdot4^{-2^m}, %
\end{eqnarray*}
%
where the last equality follows from
Lemma~\ref{predecessors-of-seeds}.
\end{pf}

\begin{pf*}{Proof of Theorem~\ref{ether-births}}
This is immediate from Lemmas \ref{blocker-ether} and
\ref{positive-probability} and the proof of Theorem~\ref{showcase}.
\end{pf*}

Recall that \textit{Extended 1 Or 3} is not diagonal- or wide-compliant.
However, it is $4$-free compliant, and this allows us to prove
the following lower bounds.

%
\begin{table}
\caption{Some nonequivalent ethers that provably emerge for \textit
{Extended 1
or 3} from a long random seed with positive asymptotic probability. Each
ether is generated from the initial condition obtained by repeating its
signature indefinitely. Here, [$k$] stands for an interval of $k$ $0$s.
The last column is a rigorous lower bound for the $\liminf$ of the probability
in Theorem \protect\ref{ether-births}. The lower bounds sum to just
over $0.826$}\label{ex1or3-table}
\begin{tabular*}{\textwidth}{@{\extracolsep{\fill}}lcccc@{}}
\hline
\textbf{Ether}&\textbf{Temporal}&\textbf{Spatial}&
\textbf{Density}&\textbf{Lower}\\
\textbf{signature}&\textbf{period}&
\textbf{period}
& \multicolumn{1}{c}{\textbf{of} $\bolds{2}$\textbf{s}}
& \textbf{bound}\\
\hline
0 & 1 & \phantom{0}1 & 0 & 0.6061 \\
02 & 1 & \phantom{0}2 & $1/2$ & 0.0471 \\
0002 & 2 & \phantom{0}4 & $1/2$ & 0.0333 \\
{[7]}2 & 4 & \phantom{0}8 & $3/8$ & 0.0664 \\
{[5]}202 & 4 & \phantom{0}8 & $3/8$ & 0.0189 \\
{[15]}2 & 8 & 16 & $5/16$ & 0.0193 \\
{[13]}202 & 8 & 16 & $11/32$ & 0.0079 \\
{[11]}20002 & 8 & 16 & $5/16$ & 0.0024 \\
{[9]}2000202 & 8 & 16 & $3/8$ & 0.0085 \\
{[9]}2020202 & 8 & 16 & $3/8$ & 0.0006 \\
{[7]}200020202 & 8 & 16 & $13/32$ & 0.0045 \\
{[7]}202020202 & 8 & 16 & $7/16$ & 0.0105 \\
{[5]}2[5]20202 & 8 & 16 & $7/16$ & 0.0006 \\
\hline
\end{tabular*}
\end{table}

\begin{theorem}[(Replication and ether probabilities for \textit
{Extended 1
Or 3})]
\label{extended-1-or-3}
Let $\xi$ be the \textit{Extended 1 Or 3} web CA, started from a uniformly
random binary seed on $[0,L]$. Then
\[
\liminf_{L\to\infty} \P(\xi\mbox{ is a replicator})\ge0.826.
\]
Moreover, lower bounds on $ \liminf_{L\to\infty} \P(\xi\mbox{ is
a replicator
with ether $\eta$}) $ for certain ethers $\eta$ are as in
Table~\ref{ex1or3-table}.
\end{theorem}

For an ether $\eta$, its \df{reflection} around the time axis is denoted
by $\bar\eta$. Then $\eta$ is \df{symmetric} if $\bar\eta$ is
equivalent to $\eta$.
Assume that $B$ is the signature of $\eta$ and $\bar B$ its
reflection.
A sufficient condition for symmetry of $\eta$ is that the reflection
$\bar B$
is a periodic shift of $B$. However, this is not a necessary condition:
the ether with signature $B=[7]200020202$ is symmetric as the fourth iteration
of the \textit{1 Or 3} rule applied on $B^\infty$ yields a translation
of $\bar B^\infty$.
Thus, the only nonsymmetric
ether in Table~\ref{ex1or3-table} is the one with signature
$[9]2000202$. In nonsymmetric cases,
our tables combine the frequencies of an ether and its reflection.

\begin{pf*}{Proof of Theorem~\ref{extended-1-or-3}}
Throughout the proof, fix a positive
integer $m\ge1$ and assume that $L\ge3\cdot2^m$.
Using the same notation as in Proposition~\ref{above-voids},
assume that for some $a\ge0$,
the configuration of $\lambda^\bullet$ at some time $t$ on $I=[a+L,
a+L-1+3\cdot2^m]\times\{t\}$ is exactly the string
$A_0=1 \lsquare1 \lsquare0 \lsquare$. For ease of reference,
we will assume $A_0^\infty$ is positioned on $\bZ$ so that $A_0$ is
the configuration in $[0, 3\cdot2^n-1]$.

Our main tool is the map $\Phi\dvtx\bZ_2^{L+1}\to\bZ_2^{3\cdot
2^m}$ that
takes as
argument an initial binary seed $\lambda_0$ supported on $[0,L]$ and
outputs the configuration of
$\lambda$ on $I$. This is a linear map that assigns to every seed with
support in $[0,L]$
its level-$2^m$ link. The matrix of $\Phi$ (in the standard basis) has
row $i$ given by the
segment $[i+2^m, L+i+2^m]$ of $A_0^\infty$, $i=0,\ldots, 3\cdot
2^m-1$. (All matrix and
vector coordinate indices start at 0.) It is easy to see that the
matrix has
rank $2^{m+1}$ and, therefore, its image has cardinality
$2^{2^{m+1}}$. The kernel of $\Phi^*$ has basis vectors
$y^k$, $k=0,\ldots, 2^m-1$, given by $y^k_i=\ind[i\mmod2^m=k]$,
$i=0,\ldots,3\cdot2^m-1$.
Therefore, the image of $\Phi$ is the set
\[
\Phi\bigl(\bZ_2^{L+1}\bigr)= \bigl\{b\in
\bZ_2^{3\cdot2^m}\dvtx b_{i}+b_{2^m+i}+b_{2^{m+1}+i}=0,
\forall i=0,\ldots, 2^m-1 \bigr\}.
\]
A vector in $\bZ_2^{3\cdot2^m}$ is naturally identified with a binary string
of length $3\cdot2^m$ and we will do so for the rest of the proof.
Let $N_n$ be the number of strings in $\Phi(\bZ_2^{L+1})$ that are
nondegenerate to depth $2^m$. Further, let $N_b$ the number of strings in
$\Phi(\bZ_2^{L+1})$ that are nondegenerate and
blockers, for $4$-free paths, to the same depth $2^m$. Observe that, for
$L$ large enough, $\Phi(\bZ_2^{L+1})$ does not depend on $L$, and
consequently neither
do $N_n$ and $N_b$.

Now suppose that $\xi_0$ is a uniform random binary seed on $[0,L]$ and
let $p_L$
be the probability that $\xi$ is a replicator.
We claim that
%
\begin{equation}
\label{pL} \liminf_{L\to\infty}p_L\ge
\frac{N_b}{N_n}.
\end{equation}

Recall that
$\pred_1$ is the event that $\xi_0$ has a predecessor; by Lemma~\ref
{predecessors-and-links},
$\pred_1^C$ is exactly the event that $\Phi(\xi_0)$ is nondegenerate
to depth $2^m$. Furthermore, conditioned on $\pred_1$, the first
predecessor of $\xi_0$ is
a uniformly random binary seed on $[1,L-1]$. Therefore,
%
%
\begin{eqnarray}
\label{pL1} %
p_L&\ge&\P\bigl(\Phi(
\xi_0)\mbox{ is a blocker to depth $2^m$}\mid\pred
_1^C\bigr) \P\bigl(\pred_1^C
\bigr)\nonumber
\\
&&{} + \P(\xi\mbox{ is a replicator}\mid\pred_1) \P(
\pred_1)
\\
&=&\frac{N_b}{N_n}\cdot\frac{3}4 + p_{L-2}\frac{1}4.\nonumber
\end{eqnarray}
Now \eqref{pL} follows by taking $\liminf$ as $L\to\infty$ of the
first and
last expressions
of~\eqref{pL1}. The particular bound was obtained by
a computer for $m=4$: all $2^{32}$ vectors in the range of $\Phi$ were
checked for blocking and nondegeneracy, and the resulting tallies
were $N_n=3,221,225,472$ and $N_b=2,663,229,504$.
This completes the proof for replication probability.

The proof for a lower bound
for a particular ether $\eta$ is identical except in the definition of
$N_b$, which is now the
number of strings in $\Phi(\bZ_2^{L+1})$ that are blockers and
nondegenerate to the
level $2^m$, \textit{and} produce $\eta$. For example, the result for
the zero ether was
$N_b=1,952,489,232$.
\end{pf*}

Table~\ref{ex1or3-table} suggests that spatial and temporal
periods of \textit{Extended 1 Or 3} ethers are powers of 2, and that
the ether
$(2)^\infty$ never appears. This is addressed in our next
two results.

%
\begin{lemma}[(Periodic configurations)] \label{ether-power-2}
Assume that $\lambda_0$ is a spatially
periodic configuration whose
period $\sigma$ divides $3\cdot2^n$, and that $\lambda_t=\lambda_0$
for some $t$. Then $\sigma$ divides $2^n$. Moreover, for $\sigma\ge1$
the temporal
period equals $\sigma/2$.
\end{lemma}

\begin{pf} By Lemmas~\ref{rescaling} and~\ref{duality},
we may assume $n=1$, and then we check
that any $\lambda_0$ of period 3 leads to a constant
configuration in a single time step. The last assertion follows from
Lemma~\ref{rescaling} and the
following two easily checked facts: (1) if $\lambda_0$ is periodic with
period at
most $2$, then $\lambda_0=\lambda_1$; and (2) if $\lambda_0$ is
periodic with period exactly $4$,
then $\lambda_0\ne\lambda_1$.
\end{pf}
%

%
%
\begin{prop}[(Possible ethers)] \label{ether-not-2}
Assume $\xi_t$ is the \textit{Extended 1 or 3} CA. Suppose that
$\xi_0$
is a replicator with ether $\eta$, and that its level-$2^m$ link is a
blocker to
depth $2^m$ for $4$-free paths. Then $\eta$ has spatial period that is a
power of 2. Also, the signature of $\eta$ is either $0$ or it is
of the form $[a_1]2[a_2]2\cdots[a_k]2$, where $k\ge1$ and each
$[a_i]$ is a string
of $0$s of odd length $a_i$. In particular, $\eta\not\equiv2$.
\end{prop}
%

\begin{pf}
The first claim follows from the previous lemma and
Theorem~\ref{above-voids-general}, so we proceed to prove the second
claim. If
$\lambda_{t}\equiv0$, but $\lambda_{t-1}\not\equiv0$, then there are,
up to
translation, exactly two possibilities for $\lambda_{t-1}$ and
$\lambda_{t-2}$:
\begin{eqnarray*}
&&\cdots1 1 1 1 0 0 1 1 1 1 0 0 \cdots\qquad \cdots0 1 0 0 0 1 0 1 0 0 0 1 \cdots
\\
&&\cdots0 1 1 0 1 1 0 1 1 0 1 1 \cdots\qquad \cdots0 1 1 0 1 1 0 1 1 0 1 1 \cdots
\end{eqnarray*}

Assume that the seed $\xi_0$ is such that $\delta(\xi_0)$ has
exactly $k$
predecessors. (Recall that $\delta(a):=\ind[a=1]$.) Then, for any
$n$, the
state of $\delta(\xi)$ on $[C-1, 2^n-C+1]\times\{2^n-k-2,2^n-k-1\}$ is
a segment
of one of the two configurations above. (Here, $C$ is a constant that depends
only on $L$.) By considering $4$-free paths, we see that the left
configuration implies $\xi_{2^n-k}\equiv0$ on $[C, 2^n-C]$, while for right
one implies that one $\xi_{2^n-k}$ vanishes outside $6\bZ\cap[C, 2^n-C]$
(after a suitable translation). As
$2$s evolve according to the \textit{1 Or 3} rule in the absence of
$1$s, the
positions of $2$s started from a subset of $2\bZ$ are a subset of
$2\bZ
$ at
all even times (by Lemma~\ref{rescaling}). The claimed form of the
signature follows.
\end{pf}

Lower bounds for ether probabilities can also be obtained for \textit
{Piggyback},
with the same proof as for Theorem~\ref{extended-1-or-3}.

%
\begin{theorem}[(Ether probabilities for \pig)]\label{piggyback-lb}
Let $\xi_t$ be the \textit{Piggyback} web CA, started from a
uniformly random
seed of $0$s and $1$s on $[0,L]$. Lower bounds on
\[
\liminf_{L\to\infty} \P(\xi\mbox{ is a replicator with ether $\eta$})
\]
are as in
Table~\ref{piggyback-table}.
\end{theorem}

%
\begin{table}
\caption{Ten ethers that emerge from long
random seeds for \textit{Piggyback} with positive asymptotic probability.
The conventions of Table \protect\ref{ex1or3-table} apply}\label
{piggyback-table}
\begin{tabular*}{\textwidth}{@{\extracolsep{\fill}}lcccc@{}}
\hline
\textbf{Ether}&\textbf{Temporal}&\textbf{Spatial}&
\textbf{Density}&\textbf{Lower}\\
\textbf{signature}&\textbf{period}&
\textbf{period}
& \multicolumn{1}{c}{\textbf{of} $\bolds{2}$\textbf{s}}
& \textbf{bound}\\
\hline
0 & 1 & \phantom{0}1 & 0 & 0.5\phantom{000} \\
2 & 1 & \phantom{0}1 & 1 & 0.0398 \\
02 & 1 & \phantom{0}2 & $1/2$ & 0.0142 \\
0002 & 2 & \phantom{0}4 & $1/2$ & 0.0258 \\
{[7]}2 & 4 & \phantom{0}8 & $3/8$ & 0.0099 \\
{[4]}2022 & 4 & \phantom{0}8 & $1/2$ & 0.0303 \\
00020222 & 4 & \phantom{0}8 & $1/2$ & 0.0209 \\
00022222 & 4 & \phantom{0}8 & $5/8$ & 0.1297 \\
0002000200022222 & 8 & 16 & $11/16$ & 0.0362 \\
0002000200202002 & 8 & 16 & $9/16$ & 0.0216 \\
\hline
\end{tabular*}
\end{table}

The computer search with $m=4$ yielded $117$ different
ethers for \textit{Piggyback}
with provably positive asymptotic probability, with their combined probabilities
at least $0.914$. The ethers listed in Table~\ref{piggyback-table}
are the ten with largest lower bounds.
The only nonsymmetric ether among these ten has signature $[4]2022$.
[The initial state
$(00020222)^\infty$ generates its translated reflection in two steps.] We
do not know whether
the asymptotic probability for the zero ether is exactly $1/2$.

\section*{Open problems}

As the earlier discussions indicate, this topic offers a rich supply of open
questions. We highlight a small selection.
\begin{longlist}[(iii)]
%
\item[(i)] In the \add\ cellular automaton started from a uniformly random binary
string on the half-line $[0,\infty)$, what is the growth rate of the
maximum integer $r_t$ for which there is an empty path from
$(-\infty,0)\times\{0\}$ to $(r_t,t)$? Is it the case that $r_t/t\to
1/4$ as $t\to\infty$?
\item[(ii)]
What can be said about
percolation in
the space--time configuration of
other one-dimensional cellular automata started in an invariant measure?
For example, the uniformly random binary string on $\bZ$ is invariant for
\emph{permutative} rules (see \cite{BL} for a definition), including
\textit{Rule 30}. Do there exist infinite diagonal, wide or empty paths?
\item[(iii)] Are there infinitely many different ethers for replicators in the
\textit{Piggyback} cellular automaton? Is there an algorithm that
decides whether a given
ether occurs in some replicator?
\item[(iv)] For two-dimensional \textit{Box 13} solidification CA (see
Section~\ref{prelim} and \cite{GG2}) started from a uniform random seed
in $[0,L]^2$, does the final configuration have rational density with
probability converging to $1$ as $L\to\infty$?
\end{longlist}

\section*{Acknowledgements}
J. Gravner gratefully acknowledges the hospitality of the Theory Group at Microsoft
Research, where most of this work was completed. We thank the referee for a very careful reading and some
useful comments.

%

\printaddresses

\end{document}